\documentclass[12pt]{amsart}
\usepackage{amsfonts}
\usepackage{fullpage}
\usepackage{amsmath,amsthm}
\usepackage{latexsym}
\usepackage{array}
\usepackage{dsfont}
\usepackage{amssymb}
\usepackage{graphics}
\usepackage{enumerate}
\usepackage[english]{babel}

\usepackage{color}
\definecolor{marin}{rgb}   {0.,   0.3,   0.7} 
\definecolor{rouge}{rgb}   {0.8,   0.,   0.} 
\definecolor{sepia}{rgb}   {0.8,   0.5,   0.} 
\usepackage[colorlinks,citecolor=marin,linkcolor=rouge,
            bookmarksopen,
            bookmarksnumbered
           ]{hyperref}

\newtheorem{lemma}{Lemma}

\newtheorem{theorem}{Theorem}
\newtheorem{proposition}{Proposition}
\newtheorem{corollary}{Corollary}
\newtheorem{remark}{Remark}
\newtheorem{example}{Example}
\newtheorem{hypothesis}{Hypothesis}
\newtheorem{notation}{Notation}
\newtheorem{definition}{Definition}
\newtheorem{conclusion}{Conclusion}
\numberwithin{equation}{section}
\newcommand{\QED}{\mbox{}\hfill \raisebox{-0.2pt}{\rule{5.6pt}{6pt}\rule{0pt}{0pt}} 
          \medskip\par}

\newcommand{\widebar}[1]{\mkern 1.5mu\overline{\mkern-1.5mu#1\mkern-1.5mu}\mkern 1.5mu}

\newcommand{\dd}{\mathrm{d}}
\newcommand{\E}{\mathbb{E}}

\newcommand{\N}{\mathbb{N}}

\newcommand{\R}{\mathbb{R}}

\newcommand{\C}{\mathbb{C}}

\newcommand{\T}{\mathbb{T}}

\newcommand{\Z}{\mathbb{Z}}

\newcommand{\er}{{\alpha}}
\newcommand{\Norm}[2]{\|#1\|\left.\vphantom{T_{j_0}^0}\!\!\right._{#2}}

\author{Erwan Faou}
\address{Univ. Rennes \& INRIA \& IRMAR  \\
ENS Rennes\\ 
Avenue Robert Schumann F-35170 Bruz, France. } 
\email{Erwan.Faou@inria.fr}

\title[Wave turbulence]
{Linearized wave turbulence convergence results \\
for three-wave systems}

\begin{document}

\begin{abstract}
We consider stochastic and deterministic three-wave semi-linear systems with bounded and almost continuous set of frequencies. Such systems can be obtained by considering nonlinear lattice dynamics or truncated partial differential equations on large periodic domains. 
We assume that the nonlinearity is small  and that the noise is small or void and acting only in the angles of the Fourier modes (random phase forcing). We consider random initial data and assume that these systems possess natural invariant distributions corresponding to some Rayleigh-Jeans stationary solutions of the wave kinetic equation appearing in wave turbulence theory. We consider random initial modes drawn with probability laws that are perturbations of theses invariant distributions. In the stochastic case, we prove that in the asymptotic limit (small nonlinearity, continuous set of frequency  and small noise), the renormalized fluctuations of the amplitudes of the Fourier modes converge in a weak sense towards the solution of the linearized wave kinetic equation around these Rayleigh-Jeans spectra. Moreover, we show that in absence of noise, the deterministic equation with the same random initial condition satisfies a generic Birkhoff reduction in a probabilistic sense, without kinetic description at least in some regime of parameters. 

\end{abstract}

\thanks{
}

\maketitle
\tableofcontents

\section{Introduction}

We consider nonlinear three-waves dynamics of the form 
\begin{equation}
\label{KPI}
\frac{\dd}{\dd t} U_k(t) = i \omega_k U_k(t) + \varepsilon Q_k^N(U,U) + g_k(\delta,U)
\end{equation}
where $k = (k_x,k_y) \in \mathcal{D}_N := \mathcal{D} \cap \Z^2/N$ where $\mathcal{D}$ is a bounded set of $\R^2$, and where $U_k(t)$ are typically renormalized Fourier coefficients of a sequence of $N$-periodic lattice points $U(t,j,\ell)$, $(j,\ell) \in \Z^2$, or of a function $U(t,x,y)$ defined on a large torus or size $N$.  The nonlinearity $Q^N$ is quadratic of convolution type and symplectic in some Hamiltonian variables. The frequency vector $\omega_k$ is the trace of a smooth real function on the discrete set $\mathcal{D}_N$ and is skew symmetric with respect to $k$. The term $g_k(\delta,U)$ models a random forcing in the angles of the complex coefficients $U_k(t)$, of strength $\delta$. We moreover make the hypothesis that \eqref{KPI} preserves the $L^2$ norm of the $U_k$, which guarantees the existence of an invariant measure corresponding to the equirepartition of energy amongst all the modes. 

As explained in the book of Nazarenko (see \cite{Naz1}) such equations are universal models for two-dimensional real nonlinear waves and can be derived as model equation for various physical phenomena such as
water waves \cite{Ablo1}, \cite{Segur}, \cite{Hammack}, but also in nonlinear optics \cite{Peli3} or plasma physics \cite{Infeld}.
Such models can be also derived from nonlinear lattice equations like  Kadomtsev--Petviashvili (KP) lattices (see for instance \cite{duncan}),   FPUT lattices (see \cite{CDRT98,BCMM15}),  finite differences or spectral approximations of real semilinear real wave equation (see for instance \cite{HL08}), or real Hamiltonian partial differential equation set of a large periodic domain of size $N$ (as in \cite{fgh}), with regularized nonlinearity truncated in frequency. 
In order to obtain fully explicit formulae and have access to simple calculations, we will retain this last modelling and require the nonlinearity to be quadratic while the linear frequencies will be fixed to 
\begin{equation}
\label{omegak}
\omega_k = k_x^3 + \eta k_y^2 k_x^{-1}, 
\end{equation}
corresponding to the frequency of the continuous KP equation.  
This choice simplifies some technical parts of the analysis concerning the resonant manifold. The parameter $\eta > 0$ is given, and will be important for the analysis of the deterministic case without random phase forcing but plays no role in the stochastic case.

In the present paper, we consider the system \eqref{KPI} in the framework of wave turbulence theory, see \cite{Peierls1}, \cite{Peierls2}, \cite{Zakh84}, \cite{ZLF92}, \cite{MMT}, \cite{LuSpohn}, \cite{Kuksin1}, \cite{Kuksin2}, \cite{Naz1}, which means that we consider random initial data, and are interested in the statistical description of the solution over long times. The external forces $g_k(\delta,U)$ models the classical
 {\em Random Phase Assumption} invoked in wave turbulence theory by a stochastic forcing in the angles of the complex coefficients $U_k(t)$.   
Such a random forcing ensures the preservation of the $L^2$, as for instance multiplicative noise in Sch\"odinger equations (see \cite{DD99}) or in conservative Hamiltonian system (see \cite{Misawa,Milstein})\footnote{It can be also compared with the convolution potentials used in KAM theory to avoid resonances (see for instance \cite{BG,ElKuk}}.

Note that the bound on the frequency set if of course a strong assumption in the model and restricts the analysis to discrete models on lattices, as in \cite{LuSpohn}, or numerical approximations of partial differential equations. However, it allows to focus  the difficulties on the {\em continuous limit} in frequency, which drives the phenomena of {\em concentration on the resonant manifold} appearing in wave turbulence.

We will consider asymptotic regimes with respect to the following three parameters: 
\begin{itemize}
\item Continuous limit in the frequency set, {\em i.e.}  $N \to \infty$. The model thus degenerates to a dispersive equation. 
\item Small nonlinearity, {\em i.e.} $\varepsilon \to 0$, the system is thus weakly non linear. 
\item Small noise: The random phase forcing is driven by independent Brownian motions in the angles of the coefficients $U_k(t)$ with variance $\delta \to 0$. 
\end{itemize}

Wave turbulence theory predicts that in a some asymptotic regimes with respect to these three parameters, then the expectations $\E \, |U_k(t)|^2$, $k = (k_x,k_y) \in \mathcal{D} \cap \Z^2/N$ are well approximated in some time scale by $r(t,k)$, $k \in \mathcal{D} \subset \R^2$ the solution of a {\em wave kinetic} equation of the form 
\begin{equation}
\label{WK}
\partial_t r(t,k) = \int_{\substack{k = j+ m \\ \omega_k = \omega_j + \omega_m}} \mathcal{Q}^k_{m j } \, r(t,j) \, r(t,m) \, \dd \Sigma(j, m) ,
\end{equation}
where $\mathcal{Q}^k_{mj}$ are real coefficients. The measure $\dd \Sigma(j,m)$ is defined as the microcanonical measure on the resonant manifold.  This kinetic equation of Boltzmann type possesses stationary solutions typically of the form $|k|^{-\alpha}$ describing Kolmogorov spectra. Amongst these solutions, Rayleigh-Jeans spectra are specific stationary solutions corresponding to equirepartition of the energy according to the invariant quantities of \eqref{KPI}. In the case of the $L^2$ norm, which is an invariant of \eqref{KPI}, it corresponds to constant stationary solutions $r(t,k) = \sigma^2$ of \eqref{WK}, for some $\sigma > 0$, which is well defined as the set of frequencies considered is here bounded.
In Hamiltonian variables - that we will use constantly in the sequel -  these distributions correspond to one-dimensional cascade solutions of the wave kinetic equation proportional to $|k_x|^{-1}$ where $k_x$ is the fourier index in the variable $x$. 


In our setting, for all given set of parameter $(N,\varepsilon,\delta)$, these particular solutions correspond to Gaussian invariant measures associated with the  $L^2$ norm of the equation. If $U_k = P_k + i Q_k$ with $P_k$ and $Q_k$ real random variables, a sampling of these measures can be done by drawing all random variables $P_k$ and $Q_k$ according to the same normal law $
\mathcal{N}(0,\sigma^2/2)$ for some $\sigma > 0$. 
For such an equirepartited initial data, the law of the solution $U_k(t)$ is stationary, and the momenta $\E\,  |U_k(t)|^2 = \sigma^2$ are constant.
 
We  consider random initial data with small perturbations of the variance $\sigma^2$ of order $\mathcal{O}(1/N^\er)$, $\er \geq 1$. Each of the modes $P_k(0)$ and $Q_k(0)$ are drawn with respect to normal laws with variance $\frac{\sigma^2}{2} + \frac{1}{2N^\er}g_0(k)$ where $g_0(k)$ are functions depending smoothly on $k$. We thus have 
\begin{equation}
\label{leonard}
\E |U_k(0)|^2 = \sigma^2 + \frac{g_0(k)}{N^\er},\quad \er \geq 1. 
\end{equation}
Under these assumptions, 
our main result can be stated as follows. We can identify the weak limit of the renormalized fluctuations,
\begin{equation}
\label{result1}
\forall\, t \in [0,T],\qquad
\lim_{\delta \to 0}   \lim_{\substack{\varepsilon \to 0 \\ N \to \infty} }  N^\er \Big( \E\, |U_k(\frac{t}{\varepsilon^2})|^2 - \sigma^2\Big) =  f(t, k) \quad \mbox{weakly,} 
\end{equation}
and prove that $f(t,k)$, $k \in \R^2$ satisfies the linearized kinetic equation on $[0,T]$
\begin{equation}
\label{WKL}
\partial_t f(t,k) = \sigma^2 \int_{\substack{k = j+ m \\ \omega_k = \omega_j + \omega_m}} \mathcal{Q}^k_{m j } \big( f(t,m) + f(t,j)\big)  \dd \Sigma(j, m) ,
\end{equation}
with $f(0,m) = g_0(m)$, 
coming from the linearization of \eqref{WK} around the stationary solution $r(t,k) = \sigma^2$. 
  Let us make some comments on this result. 
\begin{itemize} 
\item  The scaling \eqref{leonard}, in particular the condition $\alpha \geq 1$ means that we consider {\em fluctuations} around the invariant measure. It is consistent with the Feldman-H\'ajek Theorem concerning the perturbation of Gaussian measures (see \cite{daprato}) and allows to obtain bounds independent of $N$. It is also  reminiscent of fluctuation scalings used in hydrodynamics limits: see for instance \cite{BGLS} and the reference therein for the 
derivation of the linear Boltzmann equation from $N$ particules in hard sphere interaction over long times. As observed in this latter reference the {\em a priori} estimate given by this scaling (see Proposition \ref{prop2} below) ensures a propagation of chaos (here the random phase assumption) over very long times. 

\item Weak convergence here means that we have to consider local averages in $k$ to obtain strong convergence of {\em coarse-grained} quantities at a scale $h$ larger than $1/N$.
The limits in $\varepsilon$ and $N$ in \eqref{result1} then depend on this coarse graining parameter $h$, but these limits commute once $h$ is considered as fixed. The effect of $\delta$ is to regularize the resonances of the linear operator and to introduce some dissipation reflecting the effect of the noise. Therefore when $\delta$ is fixed and $(N,\varepsilon) \to (\infty,0)$, averaging can be performed and some limit identified: when $\varepsilon$ becomes small the averaged terms can be isolated, and when $N \to \infty$ this averaged term has a continuous limit allowing the concentration on the resonant manifold. 
Precise statements with explicit error estimates are given in Theorem \ref{th1} and Corollary \ref{cor1}. 
\item  
Acting only in the angle of the Fourier coefficients, the system has a degenerate noise, and can be considered as a Langevin system  with small noise in large dimension. Equation \eqref{result1} expresses the fact that the {\em normal form} term at the order $2$ (or the {\em effective equation} in the terminology of \cite{K10,K13}) inherits dissipative properties in the actions $|U_k|^2$ from random forcing in the angles of $U_k$. We can thus interpret the kinetic equation as an effect of the {\em hypoellipticity} of the system which is due to the presence of the nonlinearity as in \cite{Martin}. 
\item The choice to work with the (KP) dispersion relation \eqref{omegak} is made to allow explicit calculations. The results presented here extends to general equations of the form \eqref{KPI} under the condition that we have: (i) a bounded but almost continuous set of frequencies, (ii) existence of an explicit invariant measure, and (iii) frequencies defining non degenerate resonant manifold. Concerning this last assumption, we can observe that the resonant manifold associated with \eqref{omegak} is made of two non intersecting branches of parabola (see Equation \eqref{resmani}) and the result could possibly be extended to other similar non degenerate situations up to some technical work. The presence of singular points, like $k = 0$ which is controlled here by \eqref{orage1}, or as appearing more drastically in the case of the nonlinear cubic models might possibly be considered but certainly requires more elaborate technics, see for instance \cite{germain}. 
\end{itemize} 

The second -complementary- result of this paper   shows that the role of the noise is crucial to obtain the kinetic equation with the same random initial data. In  {\em generic} situations, {\em i.e.} for almost all $\eta >0$ (see \eqref{omegak}) we have for fixed $t \in [0,T]$, 
\begin{equation}
\label{result2}
\lim_{N \to \infty}  \lim_{\varepsilon \to 0 }\,  \lim_{\delta \to 0}  N^\er \Big( \E\, |U_k(\frac{t}{ \varepsilon^2})|^2 - \sigma^2\Big) =  g_0(k) 
\end{equation}
contradicting \eqref{result1} in this specific asymptotic regime.   This result should not be a surprise for readers familiar with normal forms or averaging theory in a deterministic setting. We can interpret it  as the fact that in the absence of random forcing ($\delta = 0$) and in the regime where $\varepsilon \to 0$ first and then $N\to \infty$, the system \eqref{KPI} admits generically in $\eta$ a weak version of Birkhoff normal form reduction preserving the actions over long times, see for instance \cite{BG} and \cite{K13} in a stochastic setting. Equation \eqref{result2} thus expresses that the normal form at the order $2$ commutes with the actions and do not change their dynamics.   It can be explained by the fact that the discrete frequencies \eqref{omegak} which are given as traces of a continuous function on the discrete grid are generically  non resonant for almost all $\eta$, with a control of the small denominators depending on $N$. We thus see that in this regime, the random forcing is crucial to obtain a kinetic representation of the dynamics. 



It is not the first occurence of a linearized kinetic equation in the mathematical literature on wave turbulence. In \cite{LuSpohn} it has been proved that in a similar setting and without noise (i.e. $\delta = 0$), if the random initial are drawn with respect to the Gibbs invariant measure associated with the complete Hamiltonian, then whereas the quantities $\mathbb{E}\, |U_k(t)|^2$ and all the moments in the Fourier space are constant in time, the evolution of space-time covariances is driven by a linearized kinetic equation.   Note that this result concerns the deterministic case, where $N \to \infty$ first and then $\varepsilon \to 0$ and uses the dispersion effects of the linear operator. 

Here we consider random initial data drawn  as fluctuations of the invariant measure at a scale $\mathcal{O}(1/N^\er)$ generating a non trivial dynamics of the actions $\mathbb{E}\, |U_k(t)|^2$ and the hierarchy of moments. 
 Our results provide an example of situation where the presence of noise yields a kinetic description of the dynamics, while the deterministic system in the same regime does not. The justification of a nonlinear kinetic dynamics even in the stochastic regime remains an open problem. 
\medskip

{\bf Acknowledgements}. This work originates from many discussions with Laure Saint-Raymond and the author would like to thank her for many advises and comments. During the preparation of the manuscript, many help was also given by Charles-Edouard Br\'ehier, Beno\^it Cadre, Arnaud Debussche, Pierre Germain and Martina Hofmanova. It is a great pleasure to thank them all. 

\section{A stochastic model}

As explained in the previous section, a system of the form \eqref{KPI} can be obtained by several modelling, from nonlinear lattices to numerical discretization and regularizations of partial differential equations. In the sequel, we make the choice of considering the stochastic KP equation
\begin{equation}
\label{bruch}
\dd U =  -\partial_{xxx} U - \eta \partial_{x}^{-1} \partial_{yy} U \, \dd t + \varepsilon  \partial_x F (U) \, \dd t  +  \sqrt{2\delta}  \, \partial_x  U \circledast \dd W (t) 
\end{equation}
where $U(t,x,y)$ is random real process depending on $(x,y) \in \T_N^2$ where $\T_N^2 = (\R / (2\pi N \Z))^2$ is a large torus ($N \in \N$, $N >> 1$) and with zero average in the direction $x$, a property that allows to define $\partial_{x}^{-1}$ and is preserved by the flow. The parameter $\eta$ is fixed and will be only used as external parameter\footnote{In the sense of KAM theory, {\em i.e.}  to generically avoid resonances, see for instance \cite{BG}, \cite{ElKuk} for applications to PDEs.} for the resonance analysis in the deterministic case (that is when $\delta = 0$). Equivalently, we could have put this parameter in the definition of the discrete  frequency grid induced by the large torus. The function $F(U)$ is a real nonlinear function of $U$ which is a regularization of the function $U \mapsto U^2$ but acting only on a bounded set of frequencies: $F(U) = P(P U)^2$ where $P$ is a smooth projector on the bounded set of frequencies. Note that at the continuous limit $N \to \infty$, the number of frequencies goes to infinity like $\mathcal{O}(N^2)$ inside this bounded set. 
The nonlinearity and the noise --which is the combination of the convolution and the Stratonovich product-- are small and represented by the numbers $\varepsilon << 1$ and $\delta << 1$. 


In the book of Nazarenko \cite{Naz1} another system, the Petviashvili model introduced in \cite{Pet1}, is also widely studied as a master model for three-wave turbulence. This equation takes the same form as \eqref{bruch} except that the linear operator is replaced by $\partial_{xxx} U + \partial_{x} \partial_{yy} U$. The analysis of this case in our setting is entirely similar to \eqref{bruch} except for the analysis of resonant manifold which is technically slightly simpler with the KP equation. Similarly, KP {\em lattices} (see \cite{duncan}) which naturally yields to model with bounded and almost continuous set of frequencies, present a similar structure with frequency involving typically trigonometric functions of $k$.   

We now describe in more detail the mathematical formulation of \eqref{bruch}. 

\subsection{Fourier coefficients and nonlinearity}
The set of frequencies associated with the two-dimensional torus $\T^2_N$ is denoted by 
$\Z_N^2 := \{ (\alpha_x,\alpha_y)/N\, | \, (\alpha_x,\alpha_y) \in \Z^2\}$. 
With a function $U(x,y)$, we associate the following normalized Fourier coefficients
$$
U_n = \frac{1}{2 \pi N} \int_{\T_N^2} e^{- i n \cdot (x,y) } U(x,y) \dd x \dd y, \quad n = (n_x,n_y) \in \Z_N^2
$$
where $n \cdot (x,y) = n_x x + n_y y$. 
With this normalization, we have for smooth enough functions
$$
U(x,y) = \frac{1}{2\pi N}\sum_{n \in \Z_N^2} e^{in\cdot( x,y)} U_n, 
$$
so that if  $U_n$ are random variable drawn with respect to a {\em random phase} measure, {\em i.e.} a measure such that $\E\,U_k \bar U_\ell =0$ for $k \neq \ell$ and $\E\,|U_k|^2 =  \sigma_k = \mathcal{O}(1)$, then we have for any measurable bounded set $B$, 
\begin{align*}
\E \int_B |U(x)|^2 \dd x &= \frac{1}{(2\pi N)^2} \sum_{k,\ell \in \Z_N^2} \int_{\T_N \cap B}e^{i(k - \ell) \cdot (x,y)} \E \, U_k \widebar U_\ell \dd x \\
&= \mathrm{Vol}(B)\frac{1}{(2\pi N)^2} \sum_{k \in \Z_N^2} \sigma_k =  \mathcal{O}(\mathrm{Vol}(B)),  
\end{align*}
uniformly in $N$ if $\sigma_k$ is sufficiently decaying and smooth (typically we will consider $\sigma_k$ as the trace on the grid $\Z_N^2$ of a smooth function with compact support in $\R^2$). 
We now define the following bounded set of frequencies. Let 
$$
\mathcal{D}^+ \subset  \{ (k_x,k_y) \in \R^2\, | \, k_x > 0\} 
$$
be a smooth bounded domain. We define $\mathcal{D}^- = \{ (-k_x,k_y) \in \R^2\, | \, k_x \in \mathcal{D}^+\}$ the symmetric of $\mathcal{D}^+$ with respect to the axis $\{ k_x = 0\}$, and we set 
$$
\mathcal{D} = \mathcal{D}^+ \cup \mathcal{D}^-.
$$
We define moreover the fine grids
$$
\mathcal{D}_N = \mathcal{D}_N^+ \cup \mathcal{D}_N^-, \quad \mathcal{D}_N^{\pm} = \mathcal{D}^{\pm} \cap \Z_N^2. 
$$
corresponding to the discrete set of frequencies. Note that $\mathrm{Card} \, \mathcal{D}_N = \mathcal{O}(N^2)$.
\begin{figure}[ht]
\begin{center}
\rotatebox{0}{\resizebox{!}{0.24\linewidth}{%
  \includegraphics{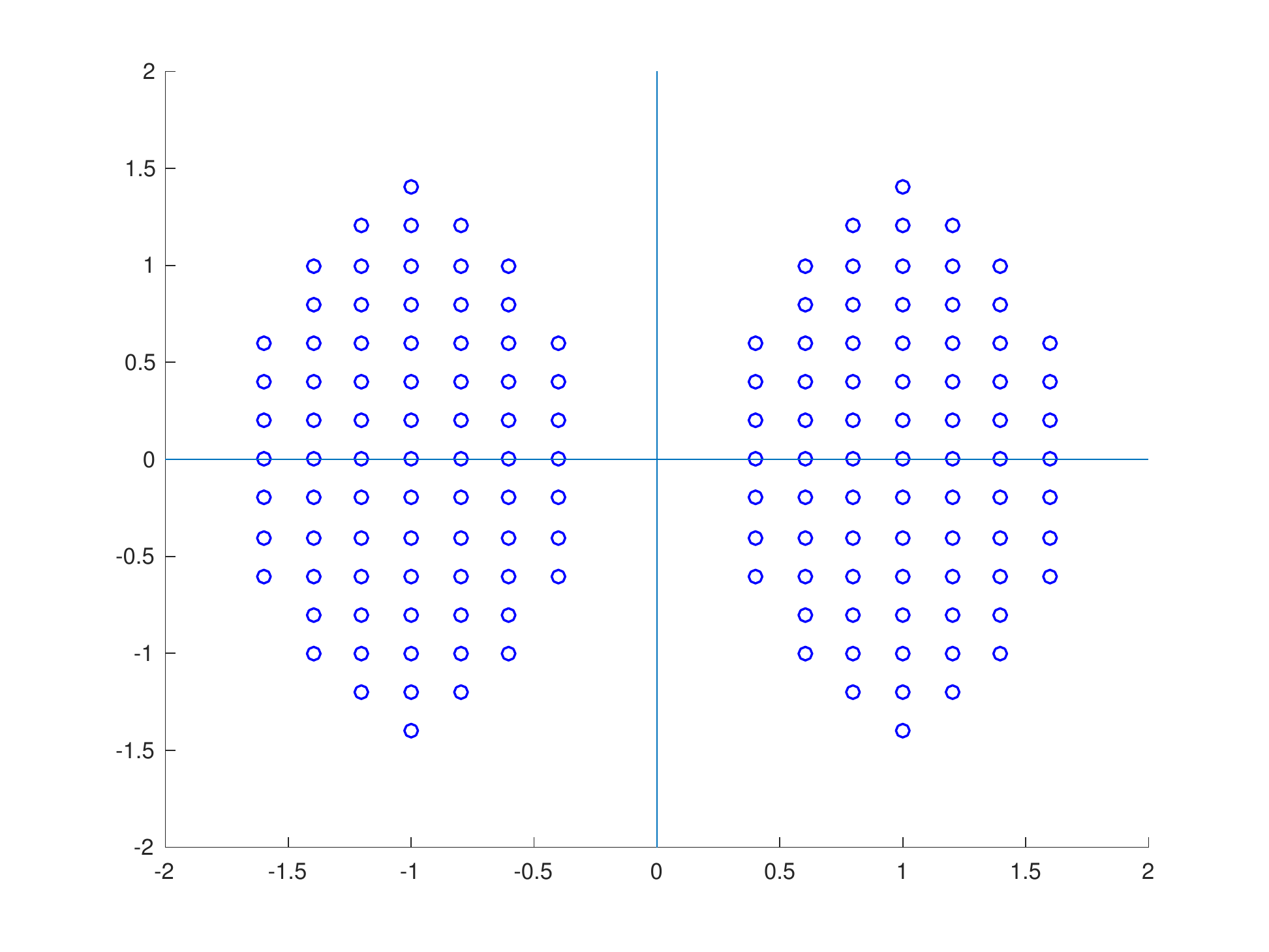}}}
  \rotatebox{0}{\resizebox{!}{0.24\linewidth}{%
  \includegraphics{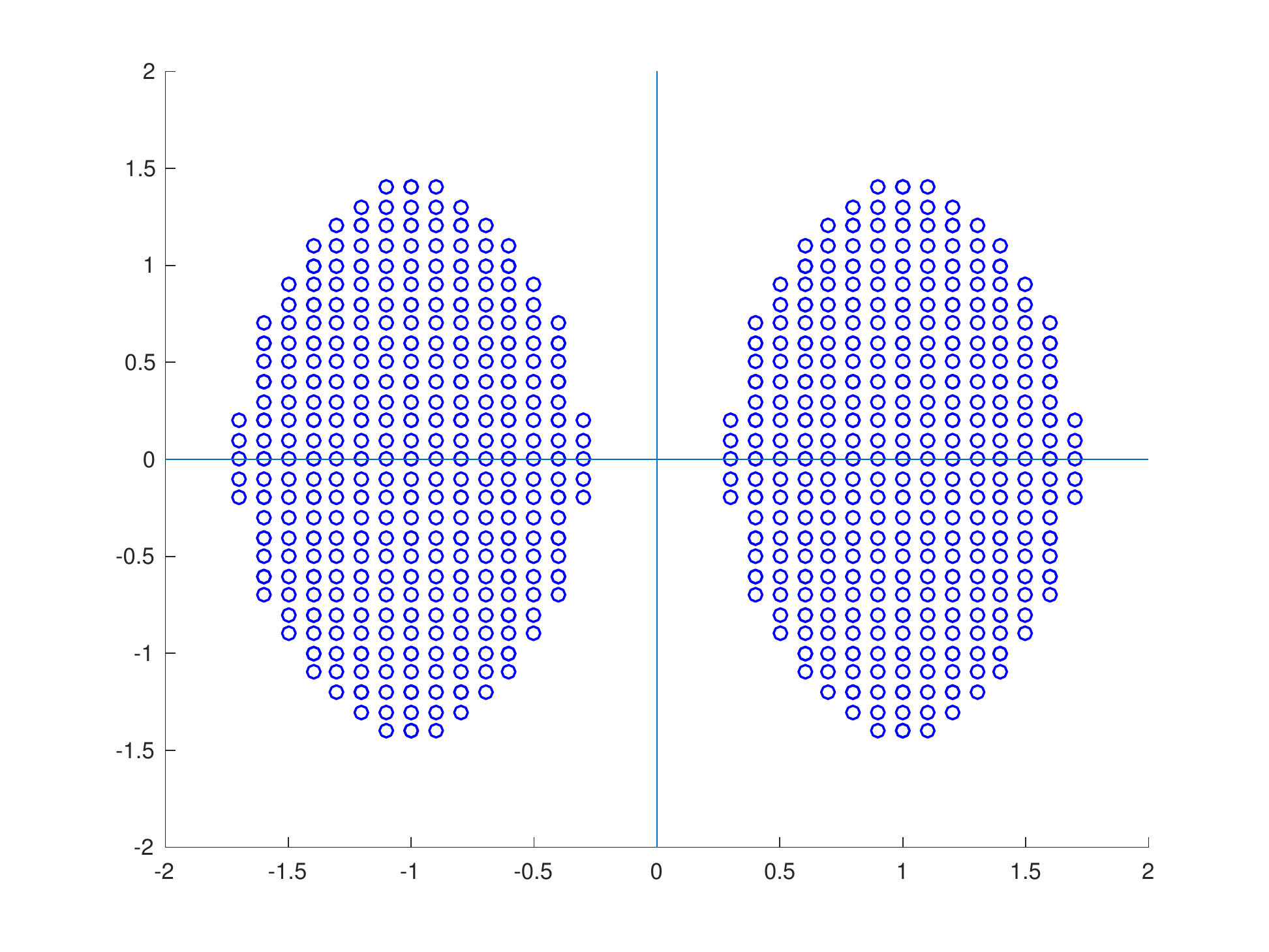}}}
  \rotatebox{0}{\resizebox{!}{0.24\linewidth}{%
  \includegraphics{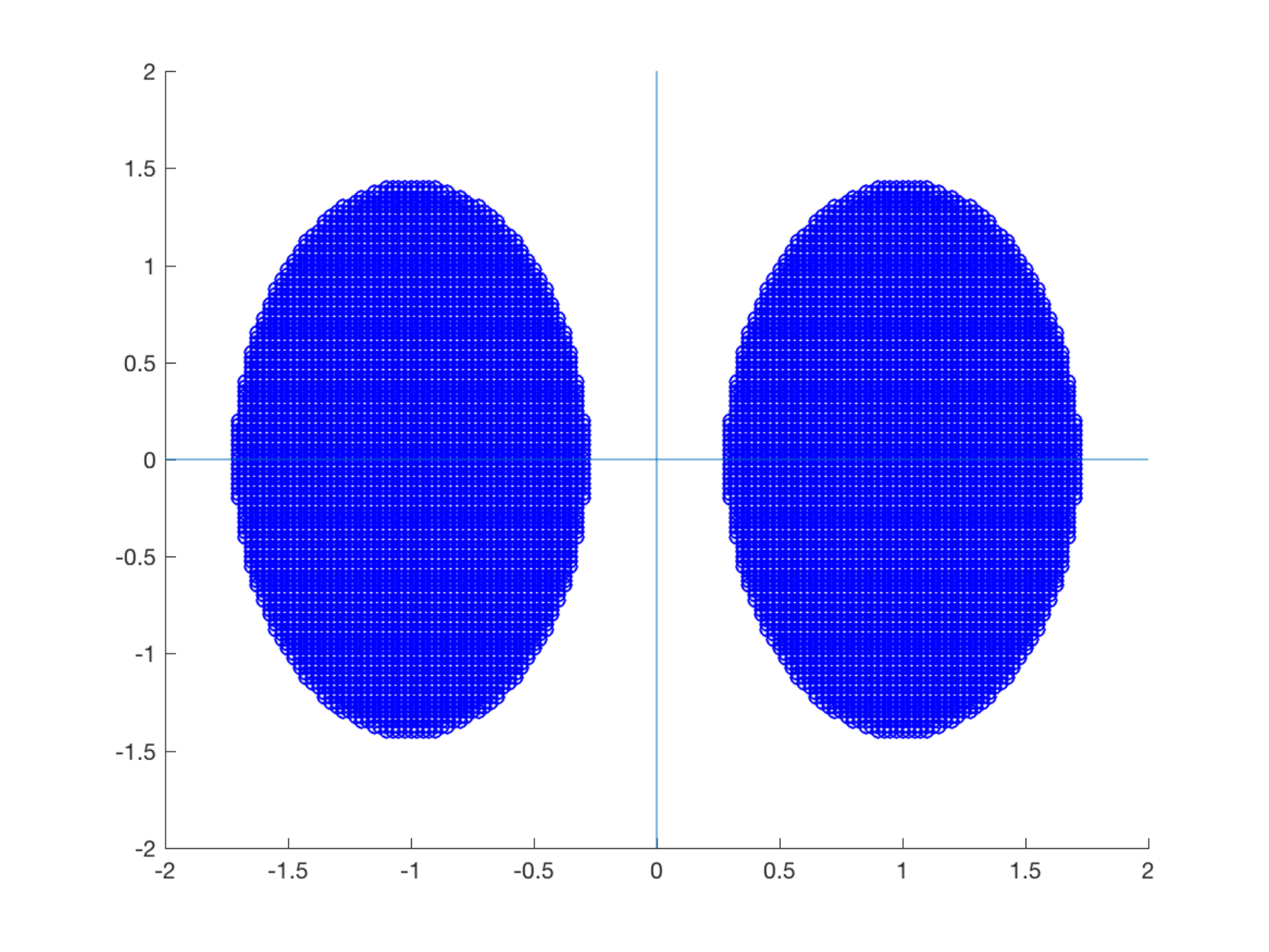}}}
  
\end{center}

\caption{Examples of sets of frequencies $\mathcal{D}_N$, for increasing values of $N$.  }
\end{figure}
 Let us fix $\psi^+:\R^2 \to \R_+$ a smooth non-negative function with support in $\mathcal{D}^+$ (a regularized indicator function of $\mathcal{D}^+$) and let $\psi = \psi^+ + \psi^{-}$ where $\psi^-(k_x,k_y) = \psi^{+}(-k_x,k_y)$.  
The projector $P$ is defined in Fourier by the formula $(P U)(k) = \psi(k) U_k$, $k \in \Z_N^2$ and  the nonlinearity is defined by the formula $F(u) = P (Pu)^2$. Hence $F(U)$ acts only on the Fourier coefficients in the set $\mathcal{D}_N$. 

\subsection{Random phase stochastic forcing}
The stochastic term 
represents a forcing in the angles of Fourier variables.  The process $W(t,x,y)$ is an $ \mathcal{O}(N^2)$ dimensional Brownian motion defined by 
$$
W(t,x,y) = \frac{1}{(2\pi N)^2} \sum_{n \in \mathcal{D}_N}  \frac{1}{ n_x} W_n(t) e^{in \cdot (x,y)},
$$
where $W_n(t)$, $n \in \mathcal{D}_N^+$, is a sequence of independent real Wiener process on a filtered probability space $(\Omega,\mathcal{F},(\mathcal{F}_t)_{t\geq 0},\mathbb{P})$ satisfying the usual conditions. For all $n \in \mathcal{D}_N$, we assume $W_{-n} = -W_{n}$ ensuring that $W(t,x,y)$ is real. The (skew-symmetric) factor $1/ n_x$ is purely convenient for the next calculations. Remark that on  $\mathcal{D}_N$, $n_x$ takes values in two finite intervals bounded away from $0$ and independent of $N$, so that at the continuous limit, the function $1/n_x$ is smooth on bounded on $\mathcal{D}$, and $W(t,x,y)$  converges essentially to a white noise when $N \to \infty$, see for instance \cite{daprato}.

 The symbol $\circledast$ means the combination of the convolution in $(x,y)$ and the Stratonovich product: 
$$
 (U \circledast \dd W(t)) (x,y)  := \int_{\T_N^2} U(t,x-x',y-y') \circ \dd W(t,x',y') \dd x'\dd y'. 
$$
With these notations, the equation in terms of the Fourier coefficients
is written
\begin{equation}
\label{mason}
\forall\, n \in \mathcal{D}_N,\qquad
\dd  U_n = i \omega_n U_n \, \dd t + i  \sqrt{2\delta} \, U_n \circ \dd W_n(t)  + i \frac{\varepsilon }{N} n_x \sum_{n = k + \ell} \chi_{nk\ell} U_k U_\ell\, \dd t, 
\end{equation} 
with $\omega_k = k_x^3 + \eta k_x^{-1} k_y^2$ which is well defined for $k \in \mathcal{D}$, and $\chi_{nk\ell} = \psi(n) \psi(k) \psi(\ell)$. As $U$ is a real function, the Fourier coefficients $U_k$ satisfy 
$$
\forall\, t \in \R, \quad \forall\, k\in  \Z_N^2 \quad U_{-k}(t) = \overline{U_k(t)}, 
$$
as can be check directly from the relations $\omega_{-n} = - \omega_n$ and  $W_{-n} = -W_n$. 
Moreover, the presence of the frequency cut-off makes the dynamics completely decoupled between the frequencies in $\mathcal{D}_N$ and the frequencies in $\Z_N^2 \backslash \mathcal{D}_N$, for which the evolution is simply given by the free evolution $e^{it \omega_n} U_n(0)$. 
In the remainder of this paper, we will thus consider the equation \eqref{mason} for $n \in \mathcal{D}_N$ only, possibly extended by natural continuity on a small neighborhood of $\mathcal{D}_N$  when needed (for the definition of coarse-grained momenta). 

Note that for $\lambda \in \R$, the scalar equation 
$$\dd U_n =  i  \lambda U_n \circ \dd W_n(t)$$ is equivalent to the It\^o equation 
$\dd U_n = i \lambda U_n\dd W - \frac{1}{2} \lambda^2 U_n \dd t$, and 
is explicitly solved. We have $$U_n(t) = e^{i \lambda W_n(t)}$$ and in particular as $W_n$ is real, $|U_n(t)|^2 = |U_n(0)|^2$. The noise is thus acting only on the angles of the Fourier coefficients and models the {\em Random Phase Assumption} that is invoked in wave turbulence (see in particular \cite{MMT}, \cite{Naz1}). 

By It\^o calculus, and using the symmetries of the equation, we can easily see that the equation \eqref{mason} possesses as natural invariant the $L^2$ norm
\begin{equation}
\label{L2norm}
M_N(U,\widebar U) := \sum_{n \in \mathcal{D}_N} |U_n|^2.  
\end{equation}
Note that as this sum is finite, this ensures the global existence of solutions of the stochastic system \eqref{mason} for any fixed deterministic initial datum. 

\subsection{Hamiltonian formulation}

Following \cite{Naz1}, we can write the KP system as a Hamiltonian system on $\mathcal{D}_N^+$. In presence of noise, that is when $\delta > 0$, we will see that it falls into the class of  stochastic systems preserving the Hamiltonian structure introduced 
in \cite{Milstein}. As the noise only acts on the angles and not on the amplitudes, we will see that up to a polar decomposition, the system admits a formulation similar to the classical Langevin equation. 
Let us make the change of unknown 
$$
V_n = \frac{U_n}{\sqrt{|n_k|}}. 
$$
The equation becomes, for $n \in \mathcal{D}_N$, 
\begin{equation}
\label{mason2}
\dd  V_n = i \omega_n V_n \, \dd t + i \sqrt{2\delta} \, V_n \circ \dd W_n(t)  +  i \frac{\varepsilon}{N} \mathrm{sign}( n_x)  \sum_{\substack{n = k + \ell \\ k,\ell\in \mathcal{D}_N}} \Psi_{nk\ell} V_k V_\ell\, \dd t. 
\end{equation} 
where $\Psi_{nk\ell} = \sqrt{|n_x||k_x||\ell_x|} \psi(n)\psi(k)\psi(\ell)$. Note that this term is symmetric in $(n,k,\ell)$ and invariant by change of sign of $n,k,\ell$. 
Note moreover that $\Psi_{nk\ell}$ is the trace on the grid $\Z_N^2$ of a smooth function defined on $\R^2$. We will by a slight abuse of notation continue to write 
$\Psi_{nk\ell}$ at the continuous limit, that is when $n,k$ and $\ell$ will take values in $\R^2$. 
As $V_{n} = \widebar V_{-n}$, we can decompose for $n \in \mathcal{D}_N^+$ the sum  
\begin{eqnarray*}
\sum_{n = k + \ell} \Psi_{nk\ell} V_k V_\ell &=& \sum_{\substack{n = k + \ell \\k_x,\ell_x > 0}} \Psi_{nk\ell}  V_k V_\ell 
+2 \sum_{\substack{n = -k + \ell \\k_x,\ell_x > 0}} \Psi_{nk\ell} \widebar V_k V_\ell\\
&=& \sum_{\substack{n = k + \ell }} \Psi_{nk\ell}^+  V_k V_\ell 
+2 \sum_{\substack{n = -k + \ell }} \Psi_{nk\ell}^+ \widebar V_k V_\ell, 
\end{eqnarray*}
where 
$
\Psi_{nk\ell}^+ = \sqrt{|n_x||k_x||\ell_x|} \psi^+(n)\psi^+(k)\psi^+(\ell)$, 
and the equation can thus be written, for $n \in \mathcal{D}_N^+$
\begin{equation}
\label{mason3}
\dd  V_n = i \omega_n V_n \, \dd t + i \sqrt{2\delta}  \, V_n \circ \dd W_n(t)  + i \frac{\varepsilon}{N}  \sum_{n = k + \ell} \Psi_{nk\ell}^+ V_k V_\ell \dd t  + 2i\frac{\varepsilon}{N}  \sum_{n = - k + \ell} \Psi_{nk\ell}^+ \widebar{V}_k V_\ell\, \dd t. 
\end{equation} 
We can write this equation under the Hamiltonian form 
$$
\forall\, n \in \mathcal{D}_N^+,\qquad 
\dd V_n = 2 i \frac{\partial H_N^\varepsilon}{\partial {\widebar{V}_n}} (V,\bar V)\dd t + i \sqrt{2\delta }\, V_n \circ \dd W_n(t), 
$$
where $V = (V_n)_{n \in \mathcal{D}_N^+}$, and where 
$
H_N^\varepsilon = \Omega_N + \varepsilon K_N  
$
with 
\begin{equation}
\label{HOK}
\Omega_N (V,\widebar{V}) = \frac12 \sum_{k \in \mathcal{D}_N^+} \omega_k |V_k|^2 \quad \mbox{and} \quad K_N (V,\widebar{V}) = \frac{1}{2N}\sum_{\substack{k+\ell - m = 0 \\k,\ell,m \in \mathcal{D}_N^+ }} \Psi_{k\ell m}^+( V_k V_\ell\widebar{V}_m + \widebar V_k \widebar V_\ell V _m). 
\end{equation}
Introducing the real processes $(P_k,Q_k)$ such that 
$$
V_k = P_k + i Q_k, \quad k \in \mathcal{D}_N^+, 
$$
we can write the Hamiltonian $H_N$ as function of the variables $(P,Q) := \{ (P_k,Q_k)\, |\, k \in \mathcal{D}_n^+\}$, and \eqref{mason3} can be written 
\begin{equation}
\label{mason5}
\left\{ 
\begin{array}{rcl}
\dd P_n &=& - \omega_n Q_n \dd t  - \sqrt{2 \delta} \, Q_n \circ \dd W_n -  \varepsilon \displaystyle \frac{\partial K_N}{\partial {Q_n}}(P,Q)\dd t, \\[2ex]
\dd Q_n &=& \omega_n P_n \dd t +  \sqrt{2 \delta} \, P_n \circ \dd W_n +  \varepsilon \displaystyle\frac{\partial K_N}{\partial {P_n}}(P,Q)\dd t.
\end{array}
\right.
\end{equation}
Note that this system falls into the framework of stochastic systems preserving the Hamiltonian structure introduced in \cite{Milstein}. 
As the $L^2$ norm of the original system is invariant, we can check that the quantities 
\begin{equation}
\label{eqinv}
M_N(P,Q) =  \sum_{k \in \mathcal{D}_N^+} \frac{1}{\gamma_k} ( P_k^2 + Q_k^2) =  \sum_{k \in \mathcal{D}_N^+} \frac{1}{\gamma_k} |V_k|^2 , \qquad  \gamma_k := \frac{1}{k_x} > 0
\end{equation}
are invariant along the trajectories on \eqref{mason3}, and correspond to the invariant $M_N(U,\widebar U)$ in the original variable, see \eqref{L2norm}. 

Eventually, let us consider the symplectic variables $(A_k,\Theta_k) \in (\R_+ \times \T)^{\mathcal{D}_N^+}$ defined by the formula 
$$
V_k = P_k + i Q_k = \sqrt{2 A_k} e^{i \Theta_k}, \quad k \in \mathcal{D}_N^+. 
$$
Then we can write the previous system under the  form 
\begin{equation}
\label{mason6}
\left\{ 
\begin{array}{rcl}
\dd A_n &=&  -  \varepsilon \displaystyle \frac{\partial K_N}{\partial {\Theta_n}}(A,\Theta) \dd t \\[2ex]
\dd \Theta_n &=& \omega_n \dd t + \sqrt{2 \delta}  \dd W_n  + \varepsilon \displaystyle\frac{\partial K_N}{\partial {A_n}}(A,\Theta)(P,Q) \dd t
\end{array}
\right.
\end{equation}
where $K_N(A,\Theta) = K_N(P,Q)$ for $(A,\Theta) = (A_k,\Theta_k)_{k \in \mathcal{D}_N^+}$. We thus see that the system is a degenerate stochastic system (the noise being active only on half of the variables) similar to the Langevin system.

\section{Invariant measures and random initial data}

We will consider the system \eqref{mason5} with random initial data $(P_k(0,\omega),Q_k(0,\omega))_{k \in \mathcal{D}_N^+}$ drawn with respect to a probability density function $\rho_N(0,p,q)$ with $$(p,q) := (p_k,q_k)_{k \in \mathcal{D}_N^+} \in (\R \times \R)^{\mathcal{D}_N^+} =: \mathcal{R}^N.$$ The function $\rho_N(0,p,q)$ is assumed to be smooth, non-negative, and such that  
$$
\int_{\mathcal{R}^N} \rho_N(0,p,q) \dd p \,\dd q = 1,
$$
where $\dd p \,\dd q := \prod_{k \in \mathcal{D}_N^+} \dd p_i\dd q_i$ is the standard Lebesgue measure on $\mathcal{R}^N$. 
It means that for a measurable domain $\mathcal{O}^N \subset \mathcal{R}^N$, we have 
$$
\mathbb{P} ( (P_k(0,\omega),Q_k(0,\omega))_{k \in \mathcal{D}_N^+} \in \mathcal{O}^N) = \int_{\mathcal{O}^N} \rho_N(0,p,q) \dd p \,\dd q. 
$$
 Using It\^o calculus, the equation for the density $\rho_N(t,p,q)$ of the probability law of the system a time $t$ is given by 
\begin{eqnarray}
\nonumber
\partial_t \rho_N &=& 
\delta L_N \rho_N    - \{ H_N^{\varepsilon}, \rho_N\}
\\\label{FKP}
&=&  \delta \sum_{k \in \mathcal{D}_N^+} ( q_k \partial_{p_k} -  p_k \partial_{q_k})^2 \rho_N  +\sum_{k \in \mathcal{D}_N^+} \omega_k ( q_k \partial_{p_k} -  p_k \partial_{q_k}) \rho_N  -\varepsilon \{ K_N, \rho_N\}, 
\end{eqnarray}
with the Poisson bracket 
$$
\{H,G\} = \sum_{k \in \mathcal{D}_N^+} \partial_{p_k} H \partial_{q_k} G - \partial_{q_k} H \partial_{p_k} G. 
$$
The operator $L_N = \sum_{k \in \mathcal{D}_N^+} ( q_k \partial_{p_k} -  p_k \partial_{q_k})^2$ can be easily interpreted in action angle variables: settting $p_k + i q_k = \sqrt{2a_k}e^{i \theta_k}$ for $(a_k,\theta_k)\in \R_+ \times \T$, $k \in \mathcal{D}_N^+$, we have $L_N = \sum_{k \in \mathcal{D}_N^+} \partial_{\theta_k}^2$ corresponding to a diffusion in the angle variables in the system \eqref{mason6}. The linear operator in 
\eqref{FKP} can thus be written $ \sum_{k \in \mathcal{D}_N^+}( - \omega_k \partial_{\theta_k} + \delta\partial_{\theta_k}^2)$ and we thus see that the noise acts as a regularization of the transport operator $  -\sum_{k \in \mathcal{D}_N^+}\omega_k \partial_{\theta_k}$ which is degenerate along the resonant manifolds. 

For a given observable $G:\mathcal{R}^N \to \C$, if $(P(t),Q(t)) = (P_k(t,\omega),Q_k(t,\omega))_{k \in \mathcal{D}_N^+}$ denotes the solution of \eqref{mason5}, we will write 
\begin{equation}
\label{brack}
\langle G(t) \rangle := \E\, G(P(t),Q(t)) = \int_{\mathcal{R}^N} G(p,q) \rho_N(t,p,q) \dd p \,\dd q. 
\end{equation}
For such a function we will write similarly 
$
\E\, G(V(t),\widebar V(t) ) 
$
the previous expectation for the function defined with an abuse of notation by $G(p,q) = G(v,\widebar v)$ with $v = (v_k)_{k \in \mathcal{D}_N^+} \in \C^{\mathcal{D}_N^+}$ satisfying $v_k = p_k + iq_k$. Typically, we will consider the evolution of momenta of the form $ \E \, V_k (t)V_m(t)\widebar V_\ell(t) $ for $k,\ell,m \in \mathcal{D}_N^+$, and the relation $V_k(t) = P_k(t) + i Q_k(t)$ will always hold implicitly. 

\subsection{Invariant measures}

Let us set
$$
\mu_N (p,q) = \frac{1}{Z_N} \exp\Big( - \sum_{k \in \mathcal{D}_N^+} \frac{1}{\gamma_k} \big(p_k^2 + q_k^2\big) \Big)
\quad
\mbox{ where }
\quad
Z_N =  \prod_{k \in \mathcal{D}_N^+} \pi \gamma_k. 
$$
We can write $\mu_N(p,q)  =(Z_N)^{-1} \exp( -  M_N(p,q) )$, where $M_N$ is the invariant function given by \eqref{eqinv}. 
Note that we could also have considered the family of  measures with probability density functions proportional to $\exp( - \beta M_N(p,q) )$ for some 
$\beta > 0$, but we fix in the remainder of the paper $\beta = 1$ for simplicity  (or equivalently with the formalism of the introduction, $\sigma = 1$ in \eqref{leonard}). 
In the sequel we will use the notation 
\begin{equation}
\label{eqnu}
\mu_N(p,q) = \prod_{k \in \mathcal{D}_N^+} \varphi(\gamma_k,p_k,q_k)\quad\mbox{where}\quad
\varphi(\gamma,x,y) := \frac{1}{\gamma\pi}e^{- \frac{1}{\gamma} (x^2 + y^2)}
\end{equation}
is the standard Gaussian probability density on $\R^2$ with variance $\gamma$. 
\begin{proposition}
The  density $\mu_N (p,q)$ is invariant by the evolution of \eqref{mason5}. We have 
$$
 L_N \mu_N = 0\quad \mbox{and} \quad \{ H_N^\varepsilon, \exp( -  M_N(p,q) )\} = 0. 
$$
so that $\mu_N$ is a stationary solution of the Fokker-Planck equation \eqref{FKP}. 
\end{proposition}
\begin{proof}
It is a direct consequence of the fact that $M_N(p,q)$ is an invariant of the system. We calculate that 
$$
\{ H_N^\varepsilon, \exp( -  M_N(p,q) )\} = -  \exp( -  M_N(p,q) ) \{ H_N^\varepsilon,M_N(p,q) \} = 0,
$$
and as  $( q_k \partial_{p_k} -  p_k \partial_{q_k}) M_N(p,q) = 0$, we have automatically $L_N \mu_N   = 0$. 
\end{proof}
Note that these densities are tensor products, and sampling according to the probability density $\rho_N(0,p,q) = \mu_N(p,q)$ is simply taking
\begin{equation}
\label{RandV}
P_k(0,\omega) \sim \mathcal{N}\big(0,\frac{\gamma_k}{2 }\big) , \qquad Q_k(0,\omega) \sim \mathcal{N}\big(0,\frac{\gamma_k}{2  }\big) , \quad k \in \mathcal{D}_N^+
\end{equation}
as independent random variables with centered Gaussian law of variance $\gamma_k/{2 }$. As the measure is invariant, any initial value starting with such condition will have the same distribution for all times. We will have in particular $\E\, G(P(t),Q(t)) = \E\, G(P(0),Q(0)) = \int_{\mathcal{R}^N} G(p,q) \mu_N(t,p,q) \dd p \,\dd q$ for all times $t \geq 0$.

As we will observe, this invariant measure is a {\em random phase} invariant measure. 
To detail this fact, we introduce some notations that will be useful for later analysis. 
\begin{notation}
\label{nota}
We will  write $\int f$ for expressions of the form $\int_{\mathcal{R}^N} f(p,q)  \dd p\, \dd q$. Moreover, 
let $\xi = (\xi_k)_{k \in \mathcal{D}_N^+}$ and $\zeta = (\zeta_k)_{k \in \mathcal{D}_N^+}$ be elements of $\N^{\mathcal{D}_N^+}$. We define $|\xi|^2 = \sum_{k \in \mathcal{D}_N^+} |\xi_k|^2$,  $\sigma(\xi) = \sum_{k \in \mathcal{D}_N^+} |\xi_k|$, 
\begin{equation*}
v^{\xi} = \prod_{k \in \mathcal{D}_N^+} v_k^{\xi_k}, \quad \mbox{and} \quad \bar v^{\zeta} = \prod_{k \in \mathcal{D}_N^+} \bar v_k^{\zeta_k},  
\end{equation*}
and similar convention for the random variable $V^\xi$ and $\widebar{V}^\zeta$. 
For a given $n \in \mathcal{D}_N^+$, we also set  
\begin{equation}
\label{eq:1}
1_n = (\delta^n_k)_{k \in \mathcal{D}_N^+},
\end{equation} where $\delta^n_k$ is the standard Kronecker symbol. 
\end{notation}

\begin{lemma}
With the previous notations, we have for $\xi$ and $\zeta$ in $\N^{\mathcal{D}_N^+}$, 
\begin{equation}
\label{orthog}
\E_{\mu_N}\, V^\xi \widebar{V}^{\zeta} = 
\int v^\xi \bar v^\zeta \mu_N = 0 \quad \mbox{if} \quad \xi \neq \zeta. 
\end{equation}
\end{lemma}
\begin{proof}
We have 
\begin{eqnarray*}
\int v^\xi \bar v^\zeta \mu_N
 &=& \frac{1}{Z_N} \prod_{k \in \mathcal{D}_N^+} \int_{\R^2} (p_k + iq_k)^{\xi_k} (p_k - iq_k)^{\zeta_k} e^{- \frac{1}{\gamma_k} (p_k^2 + q_k^2)} d p_k dq_k\\
&=& \frac{1}{Z_N} \prod_{k \in \mathcal{D}_N^+} \int_{ \R^+ \times \T} (\sqrt{2a_k})^{\xi_k + \zeta_k} e^{i \theta_k (\xi_k - \zeta_k)} e^{- 2  \frac{1}{\gamma_k}  a_k} \dd a_k \dd \theta_k, 
\end{eqnarray*}
where have have performed the symplectic change of coordinates $p_k + i q_k = \sqrt{2a_k} e^{i \theta_k}$ with $(a_k,\theta_k) \in \R^+ \times \T$. Now if 
$\xi \neq \zeta$, there exists $k$ such that $\xi_k \neq \zeta_k$ and the previous integral is zero. 
\end{proof}
For random variables satisfying \eqref{RandV}, we have using the previous result for $\xi = 1_k$ and $\zeta = 1_\ell$, 
$$
\E_{\mu_N} V_k \widebar V_\ell = 0, \qquad k \neq \ell, 
$$
and for all $k \in \mathcal{D}_N^+$
\begin{eqnarray}
\label{Eref}
\E_{\mu_N} |V_k|^2  =\E_{\mu_N} \big(|P_k|^2 + |Q_k|^2\big) &=& \int_{\mathcal{R}^N} (p_k^2 + q_k^2) \mu_N(p,q) d p \,d q \\
&=&  \frac{1}{\pi \gamma_k} \int_{\R^2} (p_k^2 + q_k^2) e^{- \frac{1}{\gamma_k} ( p_k^2 + q_k^2) {}}d p_k d q_k = \gamma_k. \nonumber 
\end{eqnarray} 
Note that if an initial value $(P_k(0), Q_k(0))$ of the system \eqref{mason5} is drawn with respect to the density probability function $\mu_N$, then the previous expectations remain constant in time. As we will see, 
\begin{equation}
\label{eqgamma}
\gamma_k = \frac{1}{ |k_x|}, \qquad k = (k_x,k_y) \in \mathcal{D}
\end{equation}
 is the Rayleigh-Jeans solution of the wave kinetic equation corresponding to the equirepartition of energy according to the $L^2$  norm of the original system, see \cite{Naz1}.  

\subsection{Random initial data}

We consider an initial probability density function $\rho_{N}(0,p,q)$  for given $N$ large enough satisfying the following hypothesis.
\begin{hypothesis}
\label{hypinit}
There exists constants $\er \geq 1$,  $C_0$ and $N_0$ and a smooth function $g_0: \mathcal{D}^+ \to \R$ such that for all $N \geq N_0$, 
\begin{equation}
\label{hypg}
\forall\, k \in \mathcal{D}_N^+\, \qquad   \E_{\rho_{N}(0)} (P_k^2 + Q_k^2) = \gamma_k +  \frac{g_0(k)}{N^\er}. 
\end{equation}
and
\begin{equation}
\label{fluct}
 \int_{\mathcal{R}^N} \frac{|\rho_{N}(0,p,q) - \mu_N(p,q)|^2}{\mu_N(p,q)} \dd p\, \dd q \leq \frac{C_0}{N^{2\er - 2}}. 
\end{equation}
\end{hypothesis}
We will construct in Section \ref{secinit} probability density functions satisfying this hypothesis for a given function $g_0$, and discuss the construction of corresponding random variables $(P_k(0,\omega),Q_k(0,\omega)) $. 
As $\mu_N$ is an invariant measure, it is easy to see that the bound \eqref{fluct} persists along the dynamics of the system. 
\begin{proposition}
\label{prop2}
Assume that  $\rho_N(0,p,q)$ satifies the bound \eqref{fluct}. Then   $\rho_N(t,p,q)$ the solution of the Fokker-Planck equation \eqref{FKP} satisfies the bound 
\begin{equation}
\label{fluct2}
\forall\, t \geq 0, \quad \int_{\mathcal{R}^N} \frac{|\rho_{N}(t,p,q) - \mu_N(p,q)|^2}{\mu_N(p,q)} \dd p \dd q \leq \frac{C_0}{N^{2\er - 2}} .
\end{equation}
\end{proposition}
\begin{proof}
We have that 
\begin{equation}
\label{eq:mkp}
\int \frac{|\rho_N(t) - \mu_N|^2}{\mu_N} = \int \frac{\rho_N(t)^2}{\mu_N}  - \int \frac{2 \rho_N(t) \mu_N}{\mu_N} + \int \frac{\mu_N^2}{\mu_N} = \int \frac{\rho_N(t)^2}{\mu_N} - 1,
\end{equation}
as $\int \rho_N(t) = 1$ for all times, which is clear from \eqref{FKP}. 
We calculate that 
\begin{align}
\label{pou}
\frac{\dd}{\dd t} \int_{\mathcal{R}^N} \frac{\rho_N^2(t,p,q)}{\mu_N(p,q)} \dd p \, \dd q  &= 2 \int_{\mathcal{R}^N} \frac{\rho_N}{\mu_N} \partial_t \rho_N \dd p\,  \dd q \nonumber \\
&= -2 \int_{\mathcal{R}^N} \frac{\rho_N}{\mu_N} \{ H_N^\varepsilon, \rho_N\} \dd p\, \dd q + 2 \delta \int_{\mathcal{R}^N} \frac{\rho_N}{\mu_N} L_N \rho_N \dd p\, \dd q.
\end{align}
Using integration by parts, 
the first term is equal to 
$$
-\int_{\mathcal{R}_N} \frac{1}{\mu_N} \{ H_N^\varepsilon, \rho_N^2\} \dd p\, \dd q  =  \int_{\mathcal{R}_N}\rho_N^2\{ H_N^\varepsilon, \frac{1}{\mu_N}\} \dd p\, \dd q = \int_{\mathcal{R}_N}\frac{\rho_N^2}{\mu_N}\{ H_N^\varepsilon, M_N\} \dd p\, \dd q = 0, 
$$
as $M_N$ is an invariant of the Hamiltonian part of the system. 
The second term in \eqref{pou} is given by 
$$
\delta \sum_{k \in \mathcal{D}_N^+} \int_{\mathcal{R}^N} \frac{\rho_N}{\mu_N} ( q_k \partial_{p_k} -  p_k \partial_{q_k})^2 \rho_N \dd p \,\dd q = - \delta \sum_{k \in \mathcal{D}_N^+} \int_{\mathcal{R}^N} \frac{1}{\mu_N} |( q_k \partial_{p_k} -  p_k \partial_{q_k}) \rho_N|^2 \dd p \,\dd q \leq 0,
$$
as we have 
$$
( q_k \partial_{p_k} -  p_k \partial_{q_k}) \frac{1}{\mu_N} = \frac{1}{\mu_N} ( q_k \partial_{p_k} -  p_k \partial_{q_k}) M_N =0. 
$$
This shows that \eqref{fluct2} is decreasing with respect to time, and hence the result. 
\end{proof}

\section{Linearized wave kinetic equation and main results}

The wave kinetic equation associated with the system \eqref{mason2} is (see \cite{Naz1})
\begin{equation}
\label{zakh}
\partial_t r(m) = 2 \int_{\substack{m = j + p \\ \omega_m = \omega_j + \omega_p}} \Psi_{mj p }^2 r(m) r(p) r(j) 
\Big( \frac{1}{r(m)} - \frac{\mathrm{sign}(m_x j_x) }{r(p)} - \frac{\mathrm{sign}(m_x p_x)}{r(j)} \Big) \dd \Sigma( j, p) 
\end{equation}
where the integral over the resonant manifold is taken with respect to the microcanonical measure induced by the Euclidean metric, see for instance \cite{Landau}. Here $r(t,m)$ with $r:\R \times \R^2 \to \R$ is defined over the whole space, and is expected to be an approximation of $\langle |V_m(t)|^2\rangle$. Similarly to the process $V_k(t)$, we verify that the right-hand side vanishes when $m \notin \mathcal{D}$.  Hence we can consider that the support of $r$ is in $\mathcal{D}$. 

We can easily verify that $r(t,m) \equiv \gamma_m =  \frac{1}{|m_x|}$ (see \eqref{eqgamma})
 is a stationary state of the kinetic equation \eqref{zakh}. In the following, we will often write $\gamma_m$ instead $\gamma(m)$ equivalently when $m \in \mathcal{D}_N$ takes values on a discrete grid, or $m \in \mathcal{D}$ is a continuous variable.  Note from \eqref{Eref} that $\gamma_m = \E_{\mu_N} |V_m|^2$ for all $N$ and all $m \in \mathcal{D}_N^+$ is the expectation of the amplitudes of the modes when they are drawn with respect to the invariant measure $\mu_N$. Hence for an initial data $V_k(0,\omega) = P_k(0,\omega) + i Q_k(0,\omega)$ drawn with respect to the invariant measure $\mu_N$ (see \eqref{RandV}) then $\langle |V_m(t)|^2\rangle = \gamma_m$, $m\in \mathcal{D}_N^+$ is indeed the trace on the grid of a constant solution of the kinetic equation \eqref{zakh}. 
 
The linearized wave kinetic equation around the stationary state $\gamma_m$ is 
\begin{equation}
\label{linzakh2}
\begin{split}
\partial_t f(m) =&  - 2 f(m) \int_{\substack{m = j + p \\ \omega_m = \omega_j + \omega_p}} \Psi_{mj p }^2 (\mathrm{sign}(m_x j_x) \gamma_j + \mathrm{sign}(m_x p_x)\gamma_p) \dd \Sigma( j, p)  \\
& + 2 \int_{\substack{m = j + p \\ \omega_m = \omega_j + \omega_p}}  \Psi_{mj p }^2 ( \gamma_j - \mathrm{sign}(m_x p_x) \gamma_m) f(p) \dd \Sigma( j, p)  \\
&+ 2 \int_{\substack{m = j + p \\ \omega_m = \omega_j + \omega_p}}  \Psi_{mj p }^2 ( \gamma_p - \mathrm{sign}(m_x j_x) \gamma_m) f(j) \dd \Sigma( j, p) 
\end{split}
\end{equation}
and we write it
\begin{equation}
\label{linzakh3}
\begin{split}
\partial_t f(m) =&   \int_{\substack{m = j + p \\ \omega_m = \omega_j + \omega_p}} (  L(m,j,p)f(m)   +  S(m,j,p) f(p) +  S(m,p,j) f(j)) \dd \Sigma( j, p)
\end{split}
\end{equation}
with 
\begin{equation}
\label{eq:LS}
\begin{array}{l}
L(m,j,p) = - 2 \Psi_{mj p }^2 (\mathrm{sign}(m_x j_x)\gamma_j + \mathrm{sign}(m_x p_x)\gamma_p) \quad\mbox{and} \\[2ex]
S(m,j,p) = 2 \Psi_{mj p }^2 ( \gamma_j - \mathrm{sign}(m_x p_x) \gamma_m). 
\end{array}
\end{equation}
As before, $f(t,m)$ can be considered as a function with support in $\mathcal{D}$ in the variable $m$. In the next section, we will show that this equation is globally well posed for initial data in $C^{1}(\mathcal{D})$ and moreover, that it is well approximated by the solution of the {\em quasi-resonant} equation defined for $\lambda >0$ by  
\begin{multline}
\label{linzakhreg}
\partial_t f_\lambda(m) =   \frac{1}{\pi}\int \frac{\lambda }{( \omega_m - \omega_{m-p} - \omega_p)^2 + \lambda^2}    L(m,m-p,p)f_\lambda(m) \dd p \\
  + \frac{1}{\pi}\int \frac{\lambda }{( \omega_m - \omega_{m-p} - \omega_p)^2 + \lambda^2}  ( S(m,m-p,p) f_\lambda(p) +  S(m,p,m-p) f_\lambda(m-p))   \dd p. 
\end{multline}

Before stating our main results, let us mention that we can rewrite the kinetic equation as an equation on $\mathcal{D}^+$ only, by assuming $r(m) = r(-m)$, and we find for $m \in \mathcal{D}^+$, 
\begin{multline}
\partial_t r(m) = 2  \int_{\substack{m = j + p \\ \omega_m = \omega_j + \omega_p}} (\Psi_{mj p }^+)^2 \big( r(j) r(p) - r(m) r(j) - r(m) r(p) \big)  \dd \Sigma(j, p)  
\\ 
+ 4 \int_{\substack{m = - j + p \\ \omega_m = - \omega_j + \omega_p}} (\Psi_{m j p }^+)^2  \big( r(j)r(p) + r(m) r(p) - r(m) r(j) \big)    \dd \Sigma(j, p), 
\end{multline}
corresponding to the formulation \eqref{mason3} (We can obviously do the same manipulation for 
the linearized kinetic equation). 

To state our main result, we need to introduce a small parameter $h >>  1/N$ and the {\em coarse} grid of mesh $h$
\begin{equation}
\label{coarsegrid}
\mathcal{G}_h = \{ (h \alpha_x,h \alpha_y) \in \mathcal{D}\, | \, (\alpha_x,\alpha_y) \in \Z^2\}. 
\end{equation}
For a given $K \in \mathcal{G}_h$, we introduce the cell 
\begin{equation}
\label{cells}
\mathcal{C}_{K}^{N,h} = \{ m \in \mathcal{D}_N\, \,|\,   K_x \leq   m_x < K_x + h, \,\,  K_y \leq   m_y < K_y + h\}, 
\end{equation}
so that 
$$
\mathrm{Card}\,  \mathcal{C}_{K}^{N,h}  = h^2 N^2. 
$$
Our main result is the following: 
\begin{theorem}
\label{th1}
Let $\eta > 0$, $\er$ with $1 \leq \er \leq 2$ be given, $T> 0$ be a fixed constant and $g_0 \in C^{1}(\mathcal{D})$ be a given function. Let $f(t,m) \in C^1(\mathcal{D})$ be the solution of the linearized kinetic equation \eqref{linzakh3} on the interval $[0,T]$ with initial data $f(0,m) = g_0(m)$, $m \in \mathcal{D}$. 
There exist constants $\varepsilon_0$, $\delta_0$, $N_0$, and $C$ such that if 
$$
\delta \leq \delta_0, \quad \varepsilon \leq \varepsilon_0,\quad N \geq N_0, \quad\mbox{and}\quad h \leq \delta^2 
$$
the following holds: First, if $f_{3\delta}(t,m)$ denote the solution of the quasi-resonant equation \eqref{linzakhreg} with $\lambda = 3 \delta$ and initial condition $f_{3\delta}(0,m) = g_0(m)$, then we have 
\begin{equation}
\label{pmlkp}
\sup_{t \in [0,T]} \sup_{m \in \mathcal{D}} | f(t,m) - f_{3\delta}(t,m)| \leq C \sqrt{\delta}. 
\end{equation}
Moreover, 
assume that $(P(0,\omega), Q(0,\omega)) = (P_k(0,\omega), Q_k(0,\omega))_{k \in \mathcal{D}_N^+}$ is a random variable of density $\rho_{N}(0,p,q)$ satisfying Hypothesis \ref{hypinit} with the function $g_0$ and perturbation scale $1/N^\er$, and let $V_k(t,\omega)$ be the solution of the system \eqref{mason2} with initial data $V_k(0,\omega) = P_k(0,\omega) + i Q_k(0,\omega)$, $k \in \mathcal{D}_N^+$ and $V_{k}(0,\omega) = \widebar{V}_{-k}(0,\omega)$ for $k \in \mathcal{D}_N^{-}$. Let us define the renormalized fluctuation observables and their averages at the scale $h$: 
\begin{equation}
\label{defluc1}
\forall\, k \in \mathcal{D}_N, \quad F_k^N (t) = N^\er \Big(\E\, |V_k(t)|^2  - \gamma_k \Big), 
\end{equation}
and the coarse-grained quantities
\begin{equation}
\label{defluc2}
\forall\, K \in \mathcal{G}_h,\quad F_K^{N,h}(t) :=  \frac{1}{h^2 N^2}\sum_{k \in \mathcal{C}^{N,h}_K } F_k^N (t), \quad\mbox{and}\quad
f_{3\delta}^{N,h}(\pi\varepsilon^2 t,K) = \frac{1}{h^2 N^2}\sum_{k \in \mathcal{C}^{N,h}_K } f_{3\delta}(\pi\varepsilon^2 t, k)
\end{equation}
where for $k \notin \mathcal{D}_N$ we set $F_k^N(t)   = 0$. 
Then we have the estimate
\begin{equation}
\label{estfund}
\forall\, t \leq \frac{T}{\pi \varepsilon^2}, \quad \sup_{K\in \mathcal{G}_h} \big|F_K^{N,h}(t) - f_{3\delta}^{N,h}(\pi\varepsilon^2 t, K) \big| \leq C \Big(  \frac{\varepsilon }{h \delta^2 }+ \frac{ 1}{ h \delta  N} +  \frac{\delta}{N^{2 - \alpha}}\Big). 
\end{equation}

\end{theorem}

This theorem shows  that when $\delta > 0$ is fixed, the limit of the renormalized fluctuations when $(N,\varepsilon) \to (\infty,0)$ is given by the quasi-resonant kinetic equation \eqref{linzakhreg} in a weak sense (at least in the case $\alpha < 2$).

By combining \eqref{pmlkp} and \eqref{estfund}, and as $f(t,m)$ is smooth (see Theorem \ref{th3} below), we obtain the following result: 
\begin{corollary}
\label{cor1}
Under the hypothesis of the previous theorem,  we have  
$$
\forall\, t \leq \frac{T}{\pi \varepsilon^2}, \quad \sup_{K\in \mathcal{G}_h} \big|F_K^{N,h}(t) - f(\pi\varepsilon^2 t, K) \big| \leq C \Big(  \frac{\varepsilon }{h \delta^2 }+ \frac{ 1}{ h \delta  N} +  \sqrt{\delta} \Big). 
$$
\end{corollary}
Essentially, these results show that for $t \in [0, T]$, in a weak sense, 
\begin{equation}
\label{eq:korber2}
\lim_{\delta \to 0} \lim_{\substack{\varepsilon \to 0 \\ N \to \infty}} \Big( N^\er \E\, |V_k(\frac{t}{\pi \varepsilon^2})|^2 -  N^\er\gamma_k \Big) =  f(t, k),  
\end{equation}
where weak convergence has to be understood as strong  convergence of local averages in $k$ over boxes of size $h$. So we have described the evolution of the fluctuations over long times and identified the first correction term in the asymptotic expansion as the solution a linearized wave kinetic equation.  We see that in this result, $\delta$ is the last parameter tending to zero, and the two limits $\varepsilon \to 0$ and $N \to \infty$ commute. 
In the previous Theorem, the role of the coarse-graining parameter $h$ is crucial to obtain the linearization of the equation. After taking the continuous limit $N \to \infty$ first, $h$ represents the size of the space averaging window to obtain the convergence in \eqref{estfund}, and is only subject to the restriction $h \leq \delta^2$ (in fact $h \leq C \delta^2$ for some constant $C$ is enough). 

The proof of this result uses strongly the bound \eqref{fluct2} which encodes some random phase effects that are propagated for all times. These effects are well expressed after averaging in space and are crucial for the linearization of the limit wave kinetic equation. The role of the small stochastic forcing is more to break the Hamiltonian structure of the equation and make the resonant kinetic equation appear as a weak effect of a hypoellipticity property of the system. 

The following result shows that in the absence of noise, the situation is completely different, and there is no kinetic description to expect at least in the  regime $\varepsilon << \frac{1}{N}$ and under some generic assumption on the coefficient $\eta$ in \eqref{bruch} (or equivalently for a generic discrete frequency grid). 
\begin{theorem}
\label{th2}
Let $T> 0$, $\er \geq 1$ be  fixed constants and $f_0 \in C^{1}(\mathcal{D})$ be a given function. Assume that $\delta = 0$ so that \eqref{mason2} is deterministic. 
For almost all parameters $\eta > 0$, there exist $\beta >0$ and constants $\varepsilon_0$, $N_0$ and $C$ such that for $h > 0$, if 
 $\varepsilon \leq \varepsilon_0$ and $N \geq N_0$, under the same hypothesis as in the previous Theorem, 
 we have the estimate 
\begin{equation}
\label{estfund2}
\forall\, t \leq \frac{T}{\varepsilon^2}\quad 
\sup_{K\in \mathcal{G}_h} \big|F_K^{N,h}(t) - g_0(K) \big| \leq C \Big(  \frac{\varepsilon N^\beta}{h }+ \frac{ 1}{ h   N} \Big). 
\end{equation}
where $g_0$ is the initial condition of Hypothesis \ref{hypinit}. 
\end{theorem}
Note that this theorem can be reformulated\footnote{strictly speaking, to obtain \eqref{estbirkhoff}, we should keep track on the dependence in $\delta \to 0$  in all the estimates. It would only yield to more technical though straightforward estimates compared to assuming directly $\delta = 0$.} as 
\begin{equation}
\label{estbirkhoff}
\lim_{N \to \infty} \lim_{\varepsilon \to 0} \, \lim_{\delta \to 0}  \Big( N^\er \E\, |V_k(\frac{t}{ \varepsilon^2})|^2 -  N^\er\gamma_k \Big) =  g_0(k),  
\end{equation}
for all $k \in \mathcal{D}_N$ and $t \in [0,T]$, 
showing the absence of non trivial kinetic description in this regime and a weak {\em preservation of the actions} in the perturbation theory terminology. This result can be interpreted as a two-steps Birkhoff normal form result  (see for instance \cite{BG}) showing preservation of the actions over long times, for some random initial data. 


%

\section{Examples of initial distributions\label{secinit}}

We now construct examples of density functions $\rho_{N}$ satisfying Hypothesis \ref{hypinit}. The simplest way to proceed is to consider small local modifications of the variance. 

\begin{proposition}
Let $g_0: \mathcal{D}^+ \to \R$ be a given function and $\alpha \geq 1$. There exists $N_0$ such that for $N \geq N_0$, the densities
$$
\rho_N(0,p,q) = \prod_{k \in \mathcal{D}_N^+} \varphi( \gamma_k + g_0(k) N^{-\alpha},p_k,q_k),
$$
where 
$\varphi(\gamma,x,y)$ is the Gaussian function \eqref{eqnu}, satisfies Hypothesis \ref{hypinit} for the function $g_0$. 
\end{proposition}
\begin{proof}
The relation \eqref{hypg} is obvious. Now with the notation $\beta_{N,k} = \gamma_k + g_0(k) N^{-\alpha}$  we calculate that 
\begin{multline*}
\int_{\mathcal{R}^N} \frac{|\rho_{N}(0,p,q)|^2}{\mu_N(p,q)} \dd p \, \dd q = \prod_{k \in \mathcal{D}_N^+} \int_{\R^2}\frac{\varphi(\beta_{N,k},p_k,q_k)^2}{\varphi(\gamma_k,p_k,q_k)} \dd p_k \, \dd q_k
\\
= \prod_{k \in \mathcal{D}_N^+} \frac{\gamma_k}{\pi\beta_{N,k}^2} \int_{\R^2}e^{  - (\frac{2}{\beta_{N,k}}- \frac{1}{\gamma_k}) (x^2 + y^2)  } \dd x \, \dd y 
=  \prod_{k \in \mathcal{D}_N^+} \frac{\gamma_k}{\beta_{N,k}^2} \frac{1}{\frac{2}{\beta_{N,k}}- \frac{1}{\gamma_k}}
= \prod_{k \in \mathcal{D}_N^+}\frac{\gamma_k^2}{\beta_{N,k} ( 2 \gamma_k - \beta_{N,k})} \\
= \prod_{k \in \mathcal{D}_N^+}\frac{\gamma_k^2}{(\gamma_k + g_0(k) N^{-\alpha}) (  \gamma_k - g_0(k) N^{-\alpha})}
= \prod_{k \in \mathcal{D}_N^+}\frac{\gamma_k^2}{\gamma_k^2 - g_0(k)^2 N^{-2\alpha}}. 
\end{multline*}
The last product is well defined for $N$ large enough, moreover, we have 
$$
\frac{\gamma_k^2}{\gamma_k^2 - g_0(k)^2 N^{-2\alpha}} = 1 + c_{N,k}\quad \mbox{with} \quad c_{N,k} = \mathcal{O}\big(\frac{1}{N^{2\alpha}}\big). 
$$
We deduce that 
\begin{equation}
\label{teleman}
\log \int_{\mathcal{R}^N} \frac{|\rho_{N}(0,p,q)|^2}{\mu_N(p,q)} \dd p \, \dd q = \sum_{k \in \mathcal{D}_N^+} \log  (1 + c_{N,k}) =: R_N \leq C (\mathrm{Card} \mathcal{D}_N^+) N^{-2 \alpha} \leq C N^{2 - 2 \alpha},
\end{equation}
for some constant $C$, and as  $\mathrm{Card} \mathcal{D}_N^+ \leq C N^2$ for some constant $C$. Hence we have that 
$$
\int_{\mathcal{R}^N} \frac{|\rho_{N}(p,q)|^2}{\mu_N(p,q)} \dd p \, \dd q = e^{R_N} = 1 + \mathcal{O}\big(\frac{1}{N^{2\alpha - 2}}\big) 
$$
as $\alpha \geq 1$, 
which shows \eqref{fluct} from \eqref{eq:mkp}. 
\end{proof}

To implement such initial condition,  if $ V_k(0,\omega) = P_k(0,\omega) + i Q_k(0,\omega)$ for $k \in \mathcal{D}_N^+$, we just draw the numbers $P_k(0,\omega)$ and $Q_k(0,\omega)$ with respect to independent normal laws $\mathcal{N}(0,\beta_{N,k})$ with $\beta_{N,k} = \gamma_k + g_0(k) N^{-\alpha}$. 

\begin{remark}
The condition $\alpha \geq 1$ ensures in the proof that $R_N$ is bounded in Eqn. \eqref{teleman}. If $\alpha \leq 1$, we get that $\int \frac{\rho_N^2}{\mu_N} = \mathcal{O}(e^{N^{2 - 2 \alpha}})$ and this prevent the linearization of the equation to occur in Section \ref{linearization} (see estimate \eqref{eqnlinearization}). A priori, larger perturbations of the invariant measure cannot be described by a linear equation and should rely on a truly  nonlinear process. 
\end{remark}

Another way to construct an initial distribution satisfying Hypothesis \ref{hypinit} for $\er = 2$ is done as follows: 
\begin{proposition}
Let $g_0: \mathcal{D}^+ \to \R$ be a given function. There exists $N_0$ and $r_0$ such that for $N \geq N_0$ and $\Norm{g_0}{L^\infty} \leq r_0$, the following holds: 
Let $|\mathcal{D}_N^+|\simeq N^2$ be the cardinal of the set $\mathcal{D}_N^+$, 
then the densities 
\begin{equation}
\label{dimanchematin}
\rho_{N}(0,p,q) = \frac{1}{|\mathcal{D}_N^+|} \sum_{k \in \mathcal{D}_N^+}  \varphi(\beta_{N,k},p_k,q_k) 
\prod_{\substack{ \ell \in \mathcal{D}_N^+ \\ \ell \neq k}}  \varphi(\gamma_\ell,p_\ell,q_\ell) , 
\end{equation}
where 
\begin{equation}
\label{betaNk}
\beta_{N,k} = \gamma_k + \frac{|\mathcal{D}_N^+| }{N^2} g_0(k), 
\end{equation}
 satisfies Hypothesis \ref{hypinit} for the function $g_0$ and $\er = 2$. 
\end{proposition}
\begin{proof}
Let $k \in \mathcal{D}_N^+$.  We first notice that as $|\mathcal{D}_N^+|\simeq N^2$, the numbers $\beta_{N,k}$ are positive for $r_0$ small enough. 
  We calculate that  
\begin{eqnarray*}
\E_{\rho_{N}(0)} (P_k^2 + Q_k^2) &=& 
\int_{\mathcal{R}^N} (p_k^2 + q_k^2) \rho_{N}(0,p,q) \dd p \,\dd q \\ &=& 
\frac{1}{|\mathcal{D}_N^+|} \gamma_k \Big(\sum_{\ell \neq k} 1\Big) +  \frac{1}{|\mathcal{D}_N^+|} \beta_{N,k}\\
&=& \gamma_k \Big( 1 - \frac{1}{|\mathcal{D}_N^+|}\Big) + \frac{1}{|\mathcal{D}_N^+| }\beta_{N,k} = \gamma_k + \frac{g_0(k)}{N^2},  
\end{eqnarray*}
so that \eqref{hypg} holds with $\er = 2$.  
Moreover, we can write 
$$
\rho_{N}(0,p,q) = \left(\frac{1}{|\mathcal{D}_N^+|} \sum_{k \in \mathcal{D}_N^+} \frac{\varphi(\beta_{N,k},p_k,q_k)}{\varphi(\gamma_k,p_k,q_k)} \right)\mu_N(p,q), 
$$
and hence 
\begin{multline*}
\int_{\mathcal{R}^N} \frac{|\rho_{N}(0,p,q)|^2}{\mu_N(p,q)} \dd p \, \dd q = \frac{1}{|\mathcal{D}_N^+|^2}\int_{\mathcal{R}^N} \left|\sum_{k \in \mathcal{D}_N^+} \frac{\varphi(\beta_{N,k},p_k,q_k)}{\varphi( \alpha_k,p_k,q_k)}  \right|^2 \mu_N(p,q) \dd p \, \dd q\\
= \frac{1}{|\mathcal{D}_N^+|^2} \sum_{k \neq \ell} \int_{\mathcal{R}^N} \frac{\varphi(\beta_{N,\ell},p_\ell,q_\ell)}{\varphi(\gamma_\ell,p_\ell,q_\ell)}\frac{\varphi(\beta_{N,k},p_k,q_k)}{\varphi( \gamma_k,p_k,q_k)} \mu_N(p,q) \dd p \, \dd q \\+ \frac{1}{|\mathcal{D}_N^+|^2}\sum_{k \in \mathcal{D}_N^+} \int_{\mathcal{R}^N}\left|\frac{\varphi(\beta_{N,k},p_k,q_k)}{\varphi( \gamma_k,p_k,q_k)}\right|^2 \mu_N(p,q) \dd p \, \dd q. 
\end{multline*}
By doing calculations similar to the ones performed in  the previous proposition, this term is equal to 
$$
\Big( 1 - \frac{1}{|\mathcal{D}_N^+|}\Big) + 
\frac{1}{|\mathcal{D}_N^+|^2}\sum_{k \in \mathcal{D}_N^+} \frac{\gamma_k^2}{\gamma_k^2 - g_0(k)^2 N^{-4}|\mathcal{D}_N^+|^2}
$$
But uniformly in $k$, we have for $N$ large enough that if $\Norm{g_0}{L^\infty} \leq r_0$ small enough, 
$$
\left|\frac{\gamma_k^2}{\gamma_k^2 - g_0(k)^2 N^{-4}|\mathcal{D}_N^+|^2}\right| \leq C
$$
for some uniform constant $C$. This shows that 
$$
\left|\int \frac{|\rho_{N}|^2}{\mu_N} - 1 \right| \leq  \frac{C}{|\mathcal{D}_N^+|}
$$
and hence the result, from \eqref{eq:mkp}. 
\end{proof}
Random variables corresponding to this initial probability density can be constructed as follows: 
Let $\epsilon = (\epsilon_j)_{j\in \mathcal{D}_N^+} \in \{0,1\}^{\mathcal{D}_N^+}$ be random variables such that 
$$
\forall\, j \in \mathcal{D}_N^+, \quad \mathbb{P}( \epsilon = 1_{j}) = \frac{1}{|\mathcal{D}_N^+|}
$$
with $1_{j} = (\delta^j_k)_{k\in \mathcal{D}_N^+} \in \{0,1\}^{\mathcal{D}_N^+}$ where $\delta_k^j$ denotes the Kronecker symbol (see \eqref{eq:1}). For $j \in \mathcal{D}_N^+$, let 
$(X^j,Y^j) = (X_k^j,Y_k^j)_{k\in \mathcal{D}_N^+} \in \mathcal{R}^N$ be random variables of probability density 
\begin{equation}
\label{probaj}
g(\beta_{N,j},p_j,q_j) \prod_{\substack{k \in \mathcal{D}_N^+ \\ k \neq j}} g(\gamma_k,p_k,q_k), 
\end{equation}
where $\beta_{N,k} = \gamma_k + \mathcal{O}( g_0(k))$ is given by \eqref{betaNk}. 
Then the random variable 
\begin{equation}
\label{inivar}
 (P(0,\omega),Q(0,\omega)) = \sum_{j \in \mathcal{D}_N^+} \epsilon_j (X^j,Y^j) \in \mathcal{R}^N
\end{equation}
has the probability density $\rho_N(p,q)$ given by \eqref{dimanchematin}.

The sampling of the initial distribution $(P(0,\omega),Q(0,\omega))$ can be done as follows: 
 One grid index $j\in \mathcal{D}^+_N$ is chosen randomly and uniformly in $\mathcal{D}^+_N$, and $(P_k(0,\omega),Q_k(0,\omega))_{k \in \mathcal{D}_N^+}$ are drawn with respect to the probability density function \eqref{probaj}. 
 
It means that once $j$ is chosen:
\begin{itemize}
\item The modes $(P_j(0,\omega),Q_j(0,\omega))$ are drawn with respect to independent normal laws $\mathcal{N}(0,\beta_{N,j}/2)$ where $\beta_{N,j} = \gamma_j + \mathcal{O}(g_0(j))$ is given by \eqref{betaNk}. 

\item For all the other modes  $k \neq j$, $(P_k(0,\omega),Q_k(0,\omega))$ are drawn with respect to independent normal laws $\mathcal{N}(0,\gamma_k/2)$. \end{itemize}

Note that this type initial condition is reminiscent of the strategy used to prove the convergence of hard-sphere dynamics to the linear Boltzmann equation, where a distribution of hard sphere is chosen randomly fixed, and the trajectory of {\em one} particle is analyzed in this scatterers environment, see for instance \cite{Gallavotti,Desvillettes,Spohn1}. Of course the situation is very different, in particular because we look from the beginning at random initial data, and all the modes are in interaction as soon as $t > 0$. 


\section{Concentration on the resonant manifold}

In this section, we define the linearized wave kinetic equation \eqref{linzakh3} and prove the existence of solutions. Morever, we prove that the solution of the regularized equation  \eqref{linzakhreg}
converges to the solution of \eqref{linzakh3} when $\lambda \to 0$. 
But before that, we need some results on quasi-resonant manifolds. For $m$ and $p$ in $\R^2$, let 
$$
\Omega(m,p) = \omega_m - \omega_{m-p} - \omega_p
$$
and for $z \in \R$, $m_x \neq 0$, let 
\begin{equation}
\label{eq:gammazm}
\Gamma(z,m) = \{p \in \R^2 \,\,  | \, \, 
\Omega(m,p) = z\}. 
\end{equation}
In order to define the resonant kinetic equation, 
we shall define the microcanonical measure on $\Gamma(z,m)$. To this aim, let us first notice that the integrand functions in \eqref{linzakhreg}, see \eqref{eq:LS},  depend on the three variables $m$, $m-p$ and $p$ and vanish if one of them is not in $\mathcal{D}$. Hence it is enough to define the measure on $\Gamma(z,m)$ for $m \in \mathcal{D}$ and $p \in \mathcal{D}$ such that $m - p \in \mathcal{D}$. 
Let $a$ and $b$ be given with $0 < a < b$. By definition of $\mathcal{D}$, we can find $a$ and $b$ such that $\mathcal{D} \subset (a,b) \times (-b,b)$, and we thus see that it will be enough to define the measure on the set of $m$ such that $a< |m_x| <b$, and for 
$$
p \in \mathcal{U}_m^a= \{ p\in \R^2, \,  |p_x| > \frac{a}{2} \quad \mbox{and}\quad  | m_x - p_x| > \frac{a}{2}\}. 
$$
Note moreover that it is enough to obtain a parametrization of $\Gamma(m,z)$ for $m_x > a$, as we notice that 
\begin{equation}
\label{eqsym}
p \in \Gamma(m,z) \Longrightarrow -p \in \Gamma(-m,-z), 
\end{equation}
as can be easily seen from the relation $\omega(-p) = - \omega(p)$. 
\begin{lemma}
\label{lemma21}
With the previous notations, let $z_0 =  \frac{3}{16} a^4$. Then for $a < m_x  < b$, and  $z \in (- z_0, z_0)$, the set
$$
\Gamma(z,m) \cap \mathcal{U}_m^a = \Gamma_a^+(z,m) \cup \Gamma_a^-(z,m)
$$
is the disjoint union of curves that can be parametrized as follows: Let 
$$
I_m:= (-\infty,-\frac{a}{2}) \cup (\frac{a}{2}, m_x - \frac{a}{2}) \cup (m_x + \frac{a}{2},\infty)
$$
and for all $\sigma\in I_m$, define the application $\kappa^{\pm}(\, \cdot\, , z,m): I_m \to \Gamma_{a}^{\pm}(z,m)$ by
$$
\kappa_x^{\pm}(\sigma,z,m) = \sigma,\quad \mbox{and} \quad \displaystyle \kappa_y^{\pm}(\sigma,z,m)   = \sigma \frac{ m_y}{m_x} \pm  \sqrt{3 \eta^{-1}} (m_x-\sigma)   \sigma \sqrt{1  - \frac{z}{3 (m_x - \sigma) m_x \sigma}}.
$$
Then we have
$$
p \in \Gamma^{\pm}_a (z,m)
 \Longleftrightarrow p = \kappa_x^{\pm}(\sigma,z,m), \quad \sigma\in I_m. 
 $$
\end{lemma}
\begin{proof}
Recall that for $k = (k_x,k_y)$ with $k_x \neq 0$, we have 
$$
\omega_k = k_x^3 + \eta \frac{k_y^2}{k_x}. 
$$
After a scaling $k_y \mapsto k_y/\sqrt{\eta}$, we see that we can reduce the analysis to the case $\eta = 1$. 

Let us first note that for a given $m$,  the condition  $p \in  \mathcal{U}_m^a$ is equivalent to $p_x \in I_m$. For such a $p_x$, 
$p = (p_x,p_y)$ is in $\Gamma(z,m)$ if and only if
$$
m_x^3 + \frac{m_y^2}{m_x} = (m_x - p_x)^3  + p_x^3 + \frac{(m_y - p_y)^2}{(m_x - p_x)} + \frac{p_y^2}{p_x} + z, 
$$
which is equivalent to 
\begin{equation}
\label{preres}
\frac{m_y^2}{m_x} = - 3 m_x^2 p_x + 3 m_x p_x^2 + \frac{(m_y - p_y)^2}{(m_x - p_x)} + \frac{p_y^2}{p_x} + z.
\end{equation}
Let us set $m_x = s p_x$, with $s \neq 0$. we have $m_x - p_x = (s -1) p_x$ with $s-1 \neq 0$ as $p_x \in I_m$. The previous relation yields 
\begin{eqnarray*}
\frac{m_y^2}{s p_x}& =& - 3 s^2 p_x^3 + 3 s p_x^3 + \frac{(m_y - p_y)^2}{(s-1)p_x} + \frac{p_y^2}{p_x} + z\\
&=& - 3 s (s-1) p_x^3 + \frac{(m_y - p_y)^2}{(s-1)p_x} + \frac{p_y^2}{p_x} + z
\end{eqnarray*}
and the equation is thus equivalent to 
\begin{eqnarray*}
3 s^2 (s-1)^2 p_x^4 &=& - (s-1) m_y^2 + s (m_y - p_y)^2 + s(s-1) p_y^2 + s(s-1)   p_x   z\\
&=& ( m_y - s p_y)^2 + s(s-1)  p_x    z
\end{eqnarray*}
which implies that 
\begin{equation}
\label{korber}
(m_y - s p_y)^2 = 3 s^2 (s-1)^2 p_x^4 - s(s-1)  p_x   z. 
\end{equation}
Hence we have 
$$
p_y = \frac{1}{s} m_y \pm \frac{1}{s} \sqrt{3 s^2 (s-1)^2 p_x^4 - s(s-1)  p_x   z} 
$$
which yields 
\begin{eqnarray*}
p_y &=& p_x \frac{ m_y}{m_x} \pm   \frac{p_x}{m_x} \sqrt{3 (m_x - p_x)^2 m_x^2  - \frac{(m_x - p_x) m_x}{  p_x  } z} \\
&=& p_x \frac{ m_y}{m_x} \pm  \sqrt{3} (m_x-p_x) p_x \sqrt{1  - \frac{z}{3 (m_x - p_x) m_x   p_x  }} 
\end{eqnarray*}
which is well defined for $z \in (-z_0,z_0)  $ given in the statement of the Lemma. 
Note that the resonant manifold  is parametrized by 
\begin{equation}
\label{resmani}
p_y = p_x \frac{ m_y}{m_x} \pm \sqrt{3} (m_x - p_x) p_x 
\end{equation}
which are branches of parabola. \end{proof}

With this result in hand, we can define the microcanonical measure on the quasi-resonant manifolds: 

\begin{definition}
Let $\phi(m,j,p)$ be a smooth function with support in $\mathcal{D}^3$. Then we set for $m \in \mathcal{D}^+$ and $z \in (-z_0,z_0)$, 
\begin{multline}
\label{defres}
\int_{\substack{m = j + p \\ \omega_m = \omega_j + \omega_p + z}}  \phi(m,j,p) \dd \Sigma( j, p) :=\\ \sum_{ \pm }\int_{I_m} \phi(m,m-\kappa^{\pm}(\sigma,z,m),\kappa^{\pm}(\sigma,z,m)) \frac{| \partial_\sigma \kappa_y^{\pm}(\sigma,z,m)|}{\Norm{\nabla_p \Omega(m,\kappa^{\pm}(\sigma,z,m))}{}}\dd \sigma
\end{multline}
and a symmetric formula for $m \in \mathcal{D}^-$ using \eqref{eqsym}, 
where $\Norm{\nabla_p \Omega(m,q)}{}$ denote the Euclidean norm of the gradient in $p$ of the function $\Omega(m,p)$ evaluated at the point $q$. 
\end{definition}
Using this explicit formula, we can easily see the following result (recall that $W^{1,\infty}(\mathcal{U})$ is the space of functions in  $L^\infty(\mathcal{U})$ with gradient in  $L^\infty(\mathcal{U})$ ). 
\begin{lemma}
\label{eqpl}
 If $\phi \in \mathcal{C}^1(\mathcal{D}^3)$ with support included in $\mathcal{D}^3$, then the application 
$$
(z,m) \mapsto \Phi(z,m) := \int_{\substack{m = j + p \\ \omega_m = \omega_j + \omega_p + z}}  \phi(m,j,p) \dd \Sigma(j, p) 
$$
is of class $\mathcal{C}^1$ on $(-z_0,z_0) \times \mathcal{D}$,  and we have 
\begin{equation}
\label{estder}
\Norm{\Phi(z,m)}{W^{1,\infty}((-z_0,z_0) \times \mathcal{D})} \leq C \Norm{\phi}{W^{1,\infty}(\mathcal{D}^3)}. 
\end{equation}
\end{lemma}
We will now gives the following result showing the approximation of the linearized wave kinetic equation by the almost-resonant equation: 
\begin{theorem}
\label{th3}
Let $g(m) \in W^{1,\infty}(\mathcal{D})$ and $T> 0$ be given. Then there exists $f(t,m) \in W^{1,\infty}(\mathcal{D})$ solution of \eqref{linzakh2} on $[0,T]$ with $f(0,m) = g(m)$. Moreover, 
there exists $\lambda_0 > 0$ such that for all $0 < \lambda \leq \lambda_0$, the equation \eqref{linzakhreg} with initial data $f_\lambda(0,m) = g(m)$ 
admits solutions in $L^{\infty}(\mathcal{D})$ in $[0,T]$ that are uniformly bounded with respect to $\lambda \in [0,\lambda_0]$. Moreover,  there exists $C$ such that 
for all $\lambda \leq \lambda_0$, we have 
\begin{equation}
\label{th3conv}
\sup_{t \in [0,T]}
\Norm{f_\lambda(t,m) - f(t,m)}{L^\infty(\mathcal{D})} \leq C \sqrt{\lambda}. 
\end{equation}
\end{theorem}
\begin{proof}
We can write the equations \eqref{linzakh2} and \eqref{linzakhreg} as 
\begin{equation}
\label{eq:L}
\partial_t f = \mathcal{L} f, \quad \mbox{and} \quad \partial_t f_\lambda = \mathcal{L}_\lambda f_\lambda. 
\end{equation}
Note that if $\Norm{\cdot }{L^\infty}$ denotes the operator norm induced by $L^\infty(\mathcal{D})$, we have 
$$
\Norm{\mathcal{L}}{L^\infty} \leq \sup_{m \in \mathcal{D}}\int_{\substack{m = j + p \\ \omega_m = \omega_j + \omega_p}} \big(| L(m,j,p)|   +  |S(m,j,p)| +  |S(m,p,j)| \big) \dd \Sigma(j,p) 
$$
and using Lemma \ref{eqpl} and the fact that the functions $| L(m,j,p)|$ and   $|S(m,j,p)|$ are smooth with compact support in $\mathcal{D}^3$, we have that $\Norm{\mathcal{L}}{L^\infty} < +\infty$ ensuring the global existence of \eqref{linzakh2} in $L^\infty(\mathcal{D})$. Moreover, by taking the gradient of \eqref{linzakh3} with respect to $m$, and using \eqref{estder}, we see that this solution is in $W^{1,\infty}(\mathcal{D})$ for all times.  

Now, we calculate that 
\begin{multline}
\label{estL}
\Norm{\mathcal{L_\lambda}}{L^\infty} \leq \\
\sup_{m \in \mathcal{D}} \frac{1}{\pi}\int \frac{\lambda }{( \omega_m - \omega_{m-p} - \omega_p)^2 + \lambda^2}   (| L(m,m-p,p)|+ | S(m,m-p,p)| + | S(m,p,m-p) | ) \dd p.
\end{multline}
To estimate this integral, we decompose the domain $\mathcal{D}$ between the set $\mathcal{U}_\lambda = \{ p \in \mathcal{D}\, |\,  | \omega_m - \omega_{m-p} - \omega_p| < \lambda^{1/4}\}$ and its complementary $\mathcal{D}\backslash\mathcal{U}_\lambda$. For $\lambda$ small enough, then the sets $\{p, \, | \,  \omega_m - \omega_{m-p} - \omega_p = z\} = \Gamma(z,m)$ are smooth submanifolds foliating $\mathcal{U}_\lambda$ for $z \in (-\lambda^{1/4},\lambda^{1/4})$. On the complementary set, the integral is simply bounded by $C \sqrt{\lambda}$ for some constant $C$ depending only on $\mathcal{D}$. 
Hence, using the co-area formula, we obtain that 
$$ 
\Norm{\mathcal{L_\lambda}}{L^\infty} \leq \\
\sup_{m \in \mathcal{D}} \frac{1}{\pi}\int_{-\lambda^{1/4}}^{+\lambda^{1/4}} 
\frac{\lambda }{z^2 + \lambda^2}   R(z,m) \dd z   + C  \sqrt{\lambda}. 
$$
where 
$$
R(z,m) = \int_{\substack{m = j + p \\ \omega_m = \omega_j + \omega_p + z}} (| L(m,j,p)|   +  |S(m,j,p)| +  |S(m,p,j)|)  \dd \Sigma(j,p)
$$
is a smooth function with compact support in $\mathcal{D}$ by the previous Lemma. 
Hence we get 
$$ 
\Norm{\mathcal{L_\lambda}}{L^\infty} \leq \\
\sup_{m \in \mathcal{D}} \frac{1}{\pi}\int_{-\lambda^{-3/4}}^{+\lambda^{-3/4}}
\frac{1}{z^2 + 1}     R(\lambda z,m) \dd z   + C \sqrt{\lambda} \leq C (\sup_{m,|z|\leq z_0} |R(z,m)| + \sqrt{\lambda}). 
$$
This shows that the equations \eqref{linzakhreg} are well posed on $L^\infty(\mathcal{D})$ with uniform bounds with respect to $\lambda \leq \lambda_0$. 
Now we have 
$$
\partial_t (f - f_\lambda) = \mathcal{L}_{\lambda}( f - f_\lambda) + (\mathcal{L} - \mathcal{L}_\lambda) f, 
$$
and the previous estimates combined with the Gronwall Lemma ensure that 
\begin{equation}
\label{tintin}
\sup_{t \in [0,T]} \Norm{f(t) - f_\lambda(t)}{L^\infty} \leq C \sup_{t \in [0,T]} \Norm{(\mathcal{L} - \mathcal{L}_\lambda) f(t)}{L^\infty}. 
\end{equation}
Now to estimate the last term, we use the same estimate as before: Using the fact that $f$ is bounded, we obtain 
$$
\Norm{(\mathcal{L} - \mathcal{L}_\lambda) f(t)}{L^\infty} \leq C \sqrt{\lambda} + \sup_{m \in \mathcal{D}} \left| \frac{1}{\pi}\int_{-\lambda^{1/4}}^{+\lambda^{1/4}} 
\frac{\lambda }{z^2 + \lambda^2}   h(t,z,m) \dd z - h(t,0,m) \right|
$$
where 
$$
h(t,z,m) = \int_{\substack{m = j + p \\ \omega_m = \omega_j + \omega_p + z}} (  L(m,j,p)f(t,m)   +  S(m,j,p) f(t,p) +  S(m,p,j) f(t,j))  \dd \Sigma(j,p). 
$$
Now as  $f(t,m)$ is uniformly bounded in $W^{1,\infty}$, we see by using the previous Lemma that $h(t,z,m)$ is Lipschitz in $(z,m)$. Hence we have 
\begin{eqnarray*}
\frac{1}{\pi}\int_{-\lambda^{1/4}}^{+\lambda^{1/4}} 
\frac{\lambda }{z^2 + \lambda^2}   h(t,z,m)  &=& \frac{1}{\pi}\int_{-\lambda^{-3/4}}^{\lambda^{-3/4}} 
\frac{1 }{z^2 + 1}   h(t,\lambda z,m) \dd z \\
&=& h(t,0,m) \left(\frac{1}{\pi}\int_{-\lambda^{-3/4}}^{+\lambda^{-3/4}} \frac{1 }{z^2 + 1}  \dd z \right) + \mathcal{O}(\lambda) \\
&=& h(t,0,m) + \mathcal{O}(\lambda^{3/4}), 
\end{eqnarray*}
which shows that 
\begin{equation}
\label{LL}
\Norm{(\mathcal{L} - \mathcal{L}_\lambda) f(t)}{L^\infty} \leq C \sqrt{\lambda}
\end{equation}
for $\lambda \leq \lambda_0$ small enough which shows the result, from \eqref{tintin}. 
\end{proof}
\begin{remark}
\label{rkderivees}
The convergence of $f_\lambda$ towards $f$ does not hold in $W^{1,\infty}$. Indeed, the evolution of the gradient of $f_\lambda$ is driven by terms of order $\mathcal{O}(\lambda^{-2})$ coming from the derivatives of the singular term $\lambda/(z^2 + \lambda^2)$. More precisely, using the estimate 
\begin{equation}
\label{traindijon}
\forall\, z \in \R,\quad \frac{\lambda |z|}{(z^2 + \lambda^2)^2} \leq \frac{1}{\lambda^2}
\end{equation} 
and the fact that $\Norm{f_\lambda}{L^\infty}$ is uniformly bounded with respect to $\lambda$
we easily obtain that 
\begin{equation}
\label{mamachou}
\sup_{t \in [0,T]}\Norm{f_\lambda(t)}{W^{1,\infty}} \leq \frac{C}{\lambda^2}, 
\end{equation}
for some constant $C$ and $\lambda$ small enough. 
\end{remark}

\section{Iterated Duhamel formula}
We now look at the evolution of the momenta $\langle V^{\xi} \widebar V^{\zeta}\rangle(t)$ for $\xi$ and $\zeta$ in $\N^{\mathcal{D}_N^+}$ (see Notation \ref{nota} and \eqref{brack} for the definition of the averages). 

Let $G(p,q) = G(v,\bar v)$ be a function of the variables $(p_k,q_k)_{k \in \mathcal{D}_N^+}$ or equivalently of the variables $(v_k,\bar v_k)_{k \in \mathcal{D}_N^+}$ after the identification $v_k = p_k + i q_k$. The evolution of $\langle G(V(t),\widebar{V}(t)) \rangle$ is given by the Kolmogorov equation (which is the adjoint of \eqref{FKP}) 
$$
\frac{\dd}{\dd t } \langle G(V,\widebar V) \rangle(t) =  \langle \{H_N^\varepsilon, G\}(V,\widebar{V}) \rangle (t) + \delta \langle (L_N G)(V ,\widebar{V}) \rangle(t), 
$$
with $H_N^\varepsilon = \Omega_N + \varepsilon K_N$ (see \eqref{HOK}). Let us take $G(v,\bar v) = v^\xi \bar v^\zeta$. 
We calculate that $(p_k \partial_{q_k} - q_k \partial_{p_k}) v_k = i p_k - q_k =  i v_k  $ and $(p_k \partial_{q_k} - q_k \partial_{p_k}) \bar v_k = - i p_k - q_k = - i \bar v_k$. Note in particular that we have $(p_k \partial_{q_k} - q_k \partial_{p_k}) (|v_k|^2) = 0$. Hence, we see that 
$$
\{\Omega_N, v^{\xi} \bar v^{\zeta}\} = \sum_{k \in \mathcal{D}_N^+} \omega_k (p_k \partial_{q_k} - q_k \partial_{p_k})  v^{\xi} \bar v^{\zeta} = i \omega\cdot  (\xi - \zeta)  v^{\xi} \bar v^{\zeta}, 
$$
where $\omega = (\omega_k)_{k \in \mathcal{D}_N^+}$ and $\omega \cdot \xi = \sum_{ k \in \mathcal{D}_N^+} \xi_k \omega_k$. 
Similarly, 
$$
L_N (v^{\xi} \bar v^{\zeta}) = - | \xi - \zeta|^2  v^{\xi} \bar v^{\zeta}. 
$$
Note that $|\xi - \zeta|^2 = \sum_{k \in \mathcal{D}_N^+} |\xi_k - \zeta_k|^2$ is essentially a measure of the number of angles in the monomial $v^{\xi}\bar v^{\zeta}$. 
Hence, we can write the equations on the momenta as 
$$
\partial_t \langle V^{\xi} \bar V^{\zeta} \rangle = (i (\xi - \zeta) \cdot \omega - \delta |\xi - \zeta|^2 ) \langle V^{\xi} \widebar V^{\zeta} \rangle + \varepsilon \sum_{\mu,\kappa}
\mathrm{Y}_{\mu\kappa} \langle V^{\mu} \widebar V^{\kappa} \rangle 
$$
where the sum is taken over multiindices $(\mu,\kappa)$ of global length $\sigma(\mu) + \sigma(\kappa) = \sigma(\xi) + \sigma(\zeta) + 1$, and where
 $\mathrm{Y}_{\mu,\kappa}$ are coefficients depending on $\Psi_{nk\ell}$ and $N$ (the precise combinatoric will be only needed for the first terms). 
 
 In the following, we set 
 $$
 \langle V^{\xi} \bar V^{\zeta} \rangle_*(t) = e^{-i (\xi - \zeta) \cdot \omega + t \delta |\xi - \zeta|^2 }  \langle V^{\xi} \widebar V^{\zeta} \rangle(t).  
 $$
 These momenta satisfy a hierarchy of equations of the form 
 \begin{equation}
 \label{mkp}
 \partial_t \langle V^{\xi} \bar V^{\zeta} \rangle_* =  \varepsilon \sum_{\mu,\kappa}e^{-i (\xi - \zeta - \mu + \kappa) \cdot \omega + t \delta ( |\xi - \zeta|^2 - |\mu - \kappa|^2) } 
\mathrm{Y}_{\mu\kappa} \langle V^{\mu} \widebar V^{\kappa} \rangle_* . 
 \end{equation}
 Note that when $\xi = \zeta$, that is when $V^{\xi} \bar V^{\zeta} = |V^\xi|^2$, we have $\langle |V^\xi|^2 \rangle_* = \langle |V^\xi|^2 \rangle$.

 We consider now the evolution of the fluctuations   (see \eqref{defluc1})
 \begin{equation}
 \label{eqFm}
 \forall\, m \in \mathcal{D}_N^+,\quad 
 F^N_m(t) := N^\er \langle |V_m|^2\rangle -  N^\er \gamma_m,
 \end{equation}
 and the coarse-grained averages (see \eqref{cells} and \eqref{defluc2})
 $$
 \forall\, K \in \mathcal{G}_h^+, \quad F^{N,h}_K(t)  = \frac{1}{h^2 N^2} \sum_{ m \in \mathcal{C}^{N,h}_K} F_m^N(t), 
 $$
 where $\mathcal{G}_h^+  = \mathcal{G}_h \cap \mathcal{D}^+$ is the intersection of the coarse grid \eqref{coarsegrid} of mesh $h$ with the right half-plane. 
Note that $F^N_m(0) = f_0(m)$ for all $m \in \Z_N^2$,  using Hypothesis \ref{hypinit}, and hence as the support of $f_0(m)$ is in $\mathcal{D}$, we can define $F_m^N(t) = 0$ for $m \notin \mathcal{D}_N$ so that the average $F_K^{N,h}(t)$ is defined without ambiguity. Moreover, as $f_0$ is smooth, we have  that 
$$
\forall\, K \in \mathcal{G}_h,\quad 
F^{N,h}_K(0) = f_0(K) + \mathcal{O}(h), 
$$
where here and in the following,  $\mathcal{O}(h)$ means here a term that is uniformly bounded by $C h$ with a constant depending only on $\mathcal{D}$ and $f_0$ (and later the final time $T$).

We first write \eqref{mkp} for $\xi = \zeta = 1_{m}$ for a given $m$, that is $\xi_k = 0$ for $k \neq m$ and $\xi_m = 1$, and $\zeta = \xi$. In this case, $V^{\xi} \bar V^{\zeta} = |V_m|^2$, and we have (see \eqref{mason3})
\begin{align}
  \partial_t \langle |V_m|^2 \rangle_* = & \mathrm{Re}\, i \frac{\varepsilon}{N}  \sum_{m = k + \ell} \Psi_{nk\ell}^+ e^{ti (-\omega_m + \omega_k + \omega_\ell) -  r_{mk\ell} \delta t}\langle \widebar V_m V_k V_\ell \rangle_*  \nonumber \\
  &+ \mathrm{Re}\, 2i\frac{\varepsilon}{N}  \sum_{m = - k + \ell} e^{ti (-\omega_m - \omega_k + \omega_\ell) -  r_{mk\ell} \delta t} \Psi_{mk\ell}^+ \langle \widebar V_m \widebar{V}_k V_\ell\rangle_*.
\end{align}
The numbers $r_{mk\ell}$ in the exponentials can only take values $3$ and $5$. 
Indeed, 
in the right-hand side, we have only terms of the form $V^\mu\widebar V^\kappa$ with $|\mu- \kappa|^2 = 3$ or $5$. For the first sum, we  have terms of the form $\widebar V_m V_k V_\ell$ which can be written $V^\mu\widebar V^\kappa$ with 
$\mu = 1_k + 1_\ell$ and $\kappa = 1_m$. We thus have $\mu- \kappa = 1_k + 1_\ell - 1_m$, under the condition $m = k + \ell$. So we have
$|\mu- \kappa|^2 = 3$ except if   $m = k
$ or $m = \ell$ or $k = \ell = m/2$.  But the first relations imply that $k$ or $\ell = 0$, and thus $\Psi_{mk\ell}^+ = 0$ by definition of $\mathcal{D}_N^+$. In the case $k = \ell$, we have $r_{mjk} = 5$. Note also that $r_{mk\ell} = r_{m\ell k}$. The analysis for the second term is similar. 
By definition of $F_m^N$, see \eqref{eqFm}, we have $\partial_t \langle |V_m|^2 \rangle_* = \frac{1}{N^\er} \partial_t F_m^N(t)$, and this relation implies that  for $K \in \mathcal{G}_h^+$, 
\begin{eqnarray}
\label{init}
F_K^{N,h}(t)  &=& F_K^{N,h}(0)
+ \int_0^t  \mathrm{Re} \, \frac{i\varepsilon N^\er}{h^2 N^{3}}   \sum_{\substack{m = k + \ell\\ m \in \mathcal{C}_K^{N,h}}} \Psi_{mk\ell}^+ e^{i s (\omega_k + \omega_\ell - \omega_m)  - r_{mk\ell}\delta s }  \langle \widebar V_m V_k V_\ell \rangle_*(s) \dd s \nonumber\\
&&
+ \int_0^t \mathrm{Re} \,   \frac{2 i\varepsilon N^\er}{h^2 N^3}    \sum_{\substack{m = - k + \ell \\ m \in \mathcal{C}_K^{N,h}}} \Psi_{mk\ell}^+ e^{i s (-\omega_k + \omega_\ell - \omega_m)  - r_{mk\ell}\delta  s } \langle \widebar V_m \widebar{V}_k V_\ell \rangle_* (s) \, \dd s. 
\end{eqnarray}
Note that the right-hand side of this equation is of order $\mathcal{O}(\frac{\varepsilon}{h} )$ using \eqref{fluct2}. Indeed, we have for all $s$, 
\begin{eqnarray*}
e^{i s (\omega_k + \omega_\ell - \omega_m)  - r_{mk\ell} s }\langle \widebar V_m V_k V_\ell \rangle_*(s) &=& \langle \widebar V_m V_k V_\ell \rangle(s)
= \int \bar v_m v_k v_\ell \rho_N(s) \\
&=& \int \bar v_m v_k v_\ell (\rho_N(s) - \mu_N), 
\end{eqnarray*}
as $\bar v_m v_k v_\ell$ is of zero average with respect to $\mu_N$ (it contains three angles). 
Hence we have by using Cauchy-Schwarz inequality
\begin{multline}
\left|\frac{\varepsilon N^\er}{h^2 N^3}\sum_{\substack{m = k + \ell\\ m\in \mathcal{C}_K^{N,h}}} \Psi_{mk\ell}^+ \langle \widebar V_m V_k V_\ell \rangle(s)\right| = \frac{\varepsilon N^\er}{h^2 N^3} \left| \int \left(\sum_{\substack{m = k + \ell\\ m\in \mathcal{C}_K^{N,h}}} \Psi_{mk\ell}^+ \bar v_m v_k v_\ell \right)(\rho_N(s) - \mu_N) \right|\\
\leq \frac{\varepsilon N^\er}{h^2 N^3}\left( \int \frac{|\rho_N(s) - \mu_N|^2}{\mu_N}\right)^{1/2}
\left(\int \left|\sum_{\substack{m = k + \ell\\ m\in \mathcal{C}_K^{N,h}}} \Psi_{mk\ell}^+ \bar v_m v_k v_\ell \right|^2\mu_N\right)^{1/2}. \label{firstest}
\end{multline}
But using \eqref{fluct2}, this term is bounded by 
$$
\frac{C \varepsilon N^\er}{h^2N^3} \times \left( \frac{C_0}{N^{2 \er - 2}} \right)^{1/2} \times \left( \sum_{\substack{m = k + \ell\\ m' = k' + \ell' \\ m, m' \in \mathcal{C}_K^{N,h}}} \widebar \Psi_{m'k'\ell'}^+\Psi_{mk\ell}^+  \int \bar v_m v_{m'}  \bar v_{k'}v_k \bar v_{\ell'}v_\ell \mu_N\right)^{1/2}. 
$$
But the terms $\int \bar v_m v_{m'}  \bar v_{k'}v_k \bar v_{\ell'}v_\ell \mu_N = 0$ unless the triplet $(m',k,\ell)$ is equal to $(m,k',\ell')$. But  we cannot have $m = k$ otherwise $\ell = 0$ and the corresponding term $\Psi_{mk\ell}^+$ cancels by definition of $\mathcal{D}$. For the same reason, we cannot have $m = \ell$, $m' = k'$ or $m' = \ell'$. Hence, to obtain a non zero term, we must have 
$k = k'$ and $\ell = \ell'$, or $k = \ell'$ and $\ell = k'$ which implies in any case that $m = m'$. Hence the previous term is bounded by 
\begin{equation}
\label{secest}
\frac{C\varepsilon}{h^2N^2} \left( 2 \sum_{\substack{m = k + \ell\\ m\in \mathcal{C}_K^{N,h}}}  \int |v_m|^2 |v_k|^2 |v_\ell|^2\mu_N\right)^{1/2} \leq \frac{C\varepsilon}{h^2N^2} (C_1 h^2 N^4)^{1/2} \leq \frac{C_2 \varepsilon}{h},  
\end{equation}
for some constant $C_1$ and $C_2$. Note that to obtain this estimate, we have used the fact that $\int |v_m|^2 |v_k|^2 |v_\ell|^2\mu_N$ are quantities uniformly bounded in $m$, $k$ and $\ell$, and that 
$$
\mathrm{Card} \{ \, (m,k,\ell) \in \mathcal{D}_N^+\, |\, m = k + \ell\quad\mbox{and}\quad m \in \mathcal{C}_K^{N,h} \,\} \leq C_1 h^2 N^4
$$
for some constant $C_1$. In the following, $C$ will denote a generic constant independent of $N$, $\delta$, $\varepsilon$, $t$ and $h$ and are allowed to change along the estimates.


\begin{proposition} Under the hypothesis \eqref{hypinit}, there exists  constants $C$, $N_0$, $h_0$ and $\delta_0$  such that for all $N\geq N_0$, $h \leq h_0$, $\delta \leq \delta_0$ and for all $t \geq 0$ and all $K \in \mathcal{G}_h^+$,   
\begin{multline}
\label{prezakh}
F_K^{N,h}(t) = F_K^{N,h}(0) 
+ 2  \int_0^t   \frac{\varepsilon^2 N^\er}{h^2 N^4}  \sum_{ \substack{ m = j+p \\ m \in \mathcal{C}_K^{N,h}} }  
 \frac{3 \delta (\Psi_{mj p}^+)^2}{ (\omega_m - \omega_j - \omega_{p})^2  + 3^2 \delta^2 } \Big(  \langle | V_j |^2 | V_p|^2  \rangle - 2\langle | V_m|^2 | V_j|^2  \rangle\Big)(s) \dd s\\
+4 \int_0^t   \frac{\varepsilon^2 N^\er}{h^2 N^4}   \sum_{\substack{ m = -j+p \\m \in \mathcal{C}_K^{N,h}}  }  
\frac{3 \delta (\Psi_{mjp}^+)^2 }{ ( \omega_m +\omega_j - \omega_{p})^2  + 3^2 \delta^2}  \Big (\langle | V_j|^2| V_p |^2\rangle + \langle | V_m|^2| V_p |^2 \rangle - \langle | V_m|^2 | V_j|^2 \rangle\Big)(s)\dd s\\
+  R(\varepsilon, K,N,h,\delta,t) 
\end{multline}
with the estimate
\begin{equation}
\label{eq:rest}
\sup_{\substack{K \in \mathcal{G}_h^+ \\ N \geq N_0}} |R(\varepsilon, K,N,h,\delta)|\leq C \Big(\frac{ \varepsilon  }{h \delta} + \frac{\varepsilon^2 }{h \delta^2}+ \frac{t \varepsilon^3 }{ h \delta^2} + \frac{t \varepsilon^2 \delta}{N^{2 - \alpha}} \Big).
\end{equation}
\end{proposition}
\begin{proof}
Let us consider the first term in the right-hand side of \eqref{init}. After integration by parts, we can write it 
\begin{multline}
\label{beurk}
\int_0^t  \mathrm{Re} \, \frac{i\varepsilon N^\er }{h^2 N^3}  \sum_{\substack{m = k + \ell\\ m \in \mathcal{C}_K^{N,h}} } \Psi_{mk\ell}^+ e^{i s (\omega_k + \omega_\ell - \omega_m)  - r_{mk\ell} s \delta}  \langle \widebar V_m V_k V_\ell \rangle_*(s) \dd s \\
= - \int_0^t  \mathrm{Re}\, \frac{\varepsilon N^\er}{h^2  N^3 }\sum_{\substack{m = k + \ell\\ m \in \mathcal{C}_K^{N,h}} } \Psi_{mk\ell}^+ \frac{e^{i s (\omega_k + \omega_\ell - \omega_m)  - r_{mk\ell} s \delta }}{ (\omega_k + \omega_\ell - \omega_m)  + i r_{mk\ell}  \delta}  \frac{\dd}{\dd t}  \langle \widebar V_m V_k V_\ell \rangle_*(s) \dd s\\
+ \left[ 
 \mathrm{Re} \, \frac{\varepsilon N^\er}{h^2 N^3}  \sum_{\substack{m = k + \ell\\ m \in \mathcal{C}_K^{N,h}}  } \Psi_{mk\ell}^+ \frac{e^{i s (\omega_k + \omega_\ell - \omega_m)  - r_{mk\ell} s \delta }}{ (\omega_k + \omega_\ell - \omega_m)  + i r_{mk\ell}  \delta}   \langle \widebar V_m V_k V_\ell \rangle_*(s) 
 \right]_0^t. 
\end{multline}
Let us look at the boundary term. By estimating like in \eqref{firstest}-\eqref{secest} and using 
\begin{equation}
\label{bounddelta}
\left| \frac{1}{(\omega_k + \omega_\ell - \omega_m)  + r_{mk\ell} i \delta}\right|^2  = \frac{1}{(\omega_k + \omega_\ell - \omega_m)^2  + r_{mk\ell}^2 \delta^2} \leq \frac{C}{ \delta^2}, 
\end{equation}
we easily obtain that 
\begin{equation}
\label{boundterm}
\left|\frac{\varepsilon N^\er}{h^2 N^3}  \sum_{\substack{m = k + \ell\\ m \in \mathcal{C}_K^{N,h}} } \Psi_{mk\ell}^+ \frac{\langle \widebar V_m V_k V_\ell \rangle(s)}{ (\omega_k + \omega_\ell - \omega_m)  + i r_{mk\ell}  \delta}    \right| \leq \frac{C\varepsilon }{ h \delta}
\end{equation}
contributing to the first term in the estimate \eqref{eq:rest}.  
Now we calculate the contribution to the time integral  term in the right-hand side of \eqref{beurk}. 
We compute 
\begin{align}
\label{beurk2}
&- \int_0^t  \mathrm{Re} \,  \frac{\varepsilon N^\er}{h^2 N^3}  \sum_{\substack{m = k + \ell\\ m \in \mathcal{C}_K^{N,h}} } \Psi_{mk\ell}^+ \frac{e^{i s (\omega_k + \omega_\ell - \omega_m)  - r_{mk\ell} s \delta}}{ (\omega_k + \omega_\ell - \omega_m)  +  ir_{mk\ell} \delta}  \frac{\dd}{\dd t}  \langle \widebar V_m V_k V_\ell \rangle_*(s) \dd s\\
&\nonumber =  \int_0^t  \mathrm{Re}\,   \frac{i\varepsilon^2 N^\er}{h^2 N^4}  \sum_{\substack{m = k + \ell \\ m = j + p \\m \in \mathcal{C}_K^{N,h}\\| 1_j + 1_p - 1_k - 1_\ell|^2 = r }  
}\Psi_{mk\ell}^+ \Psi_{m j p}^+\frac{e^{i s (\omega_k + \omega_\ell - \omega_j - \omega_p)  - r s \delta}}{ (\omega_k + \omega_\ell - \omega_{m})  +  ir_{mk\ell} \delta}  \langle \widebar V_j \widebar V_p  V_k V_\ell \rangle_*(s) \dd s\\
& \nonumber - 2 \int_0^t  \mathrm{Re} \, \frac{i\varepsilon^2 N^\er}{h^2 N^4}  \sum_{\substack{m = k  + \ell \\ k = j + p \\ m \in \mathcal{C}_K^{N,h}\\| 1_m - 1_j - 1_p - 1_\ell|^2 = r}} \Psi_{m k\ell}^+ \Psi_{k jp}^+ \frac{e^{i s (\omega_p + \omega_j + \omega_\ell - \omega_m)  - r s \delta}}{ (\omega_{k} + \omega_\ell - \omega_m)  +  ir_{mk\ell} \delta}    \langle \widebar V_m V_j V_p  V_\ell \rangle_*(s) \dd s \\
& \nonumber + 2 \int_0^t  \mathrm{Re} \,\frac{i\varepsilon^2 N^\er}{h^2 N^4} \sum_{\substack{m = k + \ell \\ m = -j + p \\m \in \mathcal{C}_K^{N,h}\\ | 1_p - 1_{j} - 1_k - 1_\ell|^2 = r}  
}\Psi_{mk\ell}^+ \Psi_{m j p}^+ \frac{e^{i s (\omega_k + \omega_\ell + \omega_j - \omega_p)  - r s \delta}}{ (\omega_k + \omega_\ell - \omega_{m})  +  i r_{mk\ell}\delta}  \langle  V_j \widebar V_p  V_k V_\ell \rangle_*(s) \dd s\\
&  \nonumber- 4 \int_0^t  \mathrm{Re}\,  \frac{i\varepsilon^2 N^\er}{h^2 N^4}  \sum_{\substack{m = k  + \ell \\ k = -j + p \\ m \in \mathcal{C}_K^{N,h}\\ |1_m + 1_{j} - 1_p - 1_\ell|^2 = r}} \Psi_{mk\ell}^+ \Psi_{k j p}^+\frac{e^{i s (\omega_p - \omega_j + \omega_\ell - \omega_m)  - r s \delta }}{ (\omega_{k} + \omega_\ell - \omega_m)  +  ir_{mk\ell} \delta}    \langle \widebar V_m \widebar V_j V_p  V_\ell \rangle_*(s) \dd s. 
\end{align}
Let us isolate the terms for which $ r = 0$, which means the terms which depends only on the $|V_k|^2$. Note that such terms can only be present in  the first and fourth terms in the right-hand side of the previous equation. 
Owing to the fact that the term $\Psi_{mk\ell}^+ = 0$ whenever one index is $0$ and is symmetric with respect to $mk\ell$, these terms are equal to 
\begin{eqnarray*}
&&  2 \int_0^t  \mathrm{Re}\,  \frac{ i\varepsilon^2 N^\er}{h^2 N^4}  \sum_{\substack{m = k +\ell\\ m \in \mathcal{C}_K^{N,h}}}  (\Psi_{mk\ell}^+)^2 \frac{1}{ (\omega_k + \omega_\ell - \omega_{m})  +  i r_{mk\ell} \delta }  \langle | V_k |^2 | V_\ell|^2  \rangle(s) \dd s\\
&& \nonumber - 4 \int_0^t  \mathrm{Re}\, \frac{ i\varepsilon^2 N^\er}{h^2 N^4}   \sum_{\substack{m = k + \ell \\ m \in \mathcal{C}_K^{N,h}} } (\Psi_{mk \ell}^+)^2 \frac{1}{ (\omega_{k} + \omega_\ell - \omega_m)  + r_{m k \ell} i \delta}    \langle | V_m|^2 | V_\ell|^2  \rangle(s) \dd s. 
\end{eqnarray*}
These terms can be written 
\begin{eqnarray}
&&  \label{popo}2 \int_0^t  \mathrm{Re}\,  \frac{ i\varepsilon^2 N^\er}{h^2 N^4}  \sum_{\substack{m = k +\ell\\ m \in \mathcal{C}_K^{N,h}}}   \frac{r_{mk\ell} \delta (\Psi_{mk\ell}^+)^2}{ (\omega_k + \omega_\ell - \omega_{m})^2  +   r_{mk\ell}^2 \delta^2 }  \langle | V_k |^2 | V_\ell|^2  \rangle(s) \dd s\\
&& \nonumber - 4 \int_0^t  \mathrm{Re}\, \frac{ i\varepsilon^2 N^\er}{h^2 N^4}   \sum_{\substack{m = k + \ell \\ m \in \mathcal{C}_K^{N,h}} } \frac{r_{mk\ell} \delta (\Psi_{mk \ell}^+)^2 }{ (\omega_{k} + \omega_\ell - \omega_m)^2  + r_{m k \ell}^2  \delta^2}    \langle | V_m|^2 | V_\ell|^2  \rangle(s) \dd s. 
\end{eqnarray}
To obtain \eqref{prezakh}, we just need to see that we can replace the $r_{mk\ell}$ to $3$ when they are equal to $5$, that is in the situation where $k = \ell = m/2$. The corresponding terms for the first sum in \eqref{popo} is 
$$
\int_0^t  \mathrm{Re}\,  \frac{ i\varepsilon^2 N^\er}{h^2 N^4}  \sum_{ m \in \mathcal{C}_K^{N,h}}   \frac{5(\Psi_{m\frac{m}{2}\frac{m}{2}}^+)^2}{ (2 \omega_{\frac{m}{2}} - \omega_{m})^2  +   5 \delta^2 }  \langle | V_\frac{m}{2} |^4  \rangle(s) \dd s. 
$$
For such terms, we have for $m_x > 0$, 
\begin{equation}
\label{orage1}
\omega_m - 2 \omega_{\frac{m}{2}} = m_x^3 + \frac{m_y^2}{m_x} - \frac{2}{8} m_x^3 -  \frac{m_y^2}{m_x} = \frac{3}{4}m_x > a
\end{equation}
for some constant $a$, 
by definition of $\mathcal{D}^+$. Hence for these terms, we can use the bound
\begin{equation}
\label{orage2}
\frac{5 \delta }{ (\omega_m - 2\omega_{\frac{m}{2}})^2  + 5^2 \delta^2 }\leq  C \delta , 
\end{equation}
provided that $\delta \leq \delta_0$ is small enough. 
Moreover, we can bound the terms 
$$
\Big|\langle | V_{\frac{m}{2}} |^4 \rangle \Big| \leq 
\left(\int  \frac{\rho_N^2} {\mu_N}\right)^{\frac12} \left(
 \int   | v_{\frac{m}{2}} |^4   \mu_N \right)^{\frac12} \leq C
$$ 
by using \eqref{fluct2} with $\er \geq 1$ and \eqref{eq:mkp}. 
We deduce that 
\begin{equation}
\label{eqnlinearization}
\left| \frac{\varepsilon^2 N^\er}{h^2N^4}
\sum_{ m \in \mathcal{C}_K^{N,h}}  
 \frac{5 \delta (\Psi_{m \frac{m}{2} \frac{m}{2} })^2}{ (\omega_m - 2\omega_\frac{m}{2} )^2  + 5^2 \delta^2 }  \langle | V_{\frac{m}{2}} |^4 \rangle 
 \right| \leq \frac{C \delta \varepsilon^2 N^\er}{h^2 N^4} \left(\sum_{ m \in \mathcal{C}_K^{N,h}} 1 \right)  \leq C \frac{\varepsilon^2 \delta}{N^{2 - \alpha}}. 
\end{equation}
Hence, up to a term which is less than the 
the last contribution in \eqref{eq:rest}, we see that we can set $r_{mk\ell}$ to $3$ everywhere in \eqref{popo}
which yields the first contribution in Equation \eqref{prezakh}.  \\

\medskip 

Going back to \eqref{beurk2}, 
we see that the other terms for which $r \geq 1$ can be integrated by part again. 
After scaling back to the original moments $\langle v^{\xi} \bar v^{\zeta} \rangle(t) = e^{i (\xi - \zeta) \cdot \omega - t \delta |\xi - \zeta|^2 }  \langle v^{\xi} \bar v^{\zeta} \rangle_*(t)$,  the new time integrals obtained are of the following types (up to complex conjugate): 
\begin{equation}
\label{h1}
\int_0^t \frac{\varepsilon^3 N^\er}{h^2 N^5} \sum_{(p,j,k,\ell,q) \in  \mathcal{A}_K^{N,h}} W_{pqjk\ell}(\delta)   \langle \widebar V_{p}
V_j V_k V_\ell V_q \rangle (s) \dd s
\end{equation}
or
\begin{equation}
\label{h2}
 \int_0^t \frac{\varepsilon^3 N^\er}{ h^2 N^5} \sum_{(p,j,k,\ell,q) \in  \mathcal{A}_K^{N,h}} W_{pqjk\ell}(\delta)   \langle \widebar V_{p}
\widebar V_j V_k V_\ell V_q \rangle(s)  \dd s
\end{equation}
where $\mathcal{A}_{K}^{N,h}$ are sets of multindex of cardinal $\mathcal{O}(h^2N^8)$ (typically for the first term, sets of multi-indices with constraints of the form $p = j + k + \ell + q$ with $p \in \mathcal{C}_K^{N,h}$) , and where 
\begin{equation}
\label{bjl}
\sup_{pqjk\ell}|W_{pqjk\ell}(\delta)| \leq \frac{C}{\delta^2}. 
\end{equation}
Note that in this case, as we have always products of five $V_k$ and $\widebar V_\ell$, we cannot have a term depending only on the $|V_k|^2$, and hence the terms are or zero average with respect to the density $\mu_N$. Hence, the integrand term \eqref{h1} for example is equal to 
$$
\int \frac{\varepsilon^3 N^\er}{h^2N^5} \sum_{(p,j,k,\ell,q) \in  \mathcal{A}_m^N} W_{pqjk\ell}(\delta)  \widebar v_{p}
v_j v_k v_\ell v_q (\rho_N(s) - \mu_N)
$$
And we can bound it in modulus by 
$$
 \frac{\varepsilon^3 N^\er}{ h^2N^5} 
\left( \int  \frac{|\rho_N(s) - \mu_N|^2}{\mu_N}\right)^{1/2}
\left( \int \left| \sum_{(p,j,k,\ell,q) \in  \mathcal{A}_K^{N,h} } W_{pjk\ell m}(s,\delta)  \bar v_{p}
v_j v_k v_\ell v_m
\right|^2  \mu_N\right)^{1/2}
$$
and using \eqref{orthog}, \eqref{bjl} and \eqref{fluct2} this term will be smaller than 
$$
\frac{C \varepsilon^3}{h^2 \delta^2 N^4}\left( \sum_{\substack{ (p,j,k,\ell,q) \in  \mathcal{A}_K^{N,h} \\ (p',j',k',\ell',q') \in  \mathcal{A}_K^{N,h}   \\ 1_p + 1_{j'} + 1_{k'} + 1_{\ell'} = 1_{p'} + 1_j + 1_k + 1_\ell}} \int| v_{p}
v_j v_k v_\ell v_m|^2
 \mu_N  \right)^{1/2}. 
$$
In this sum, the integral with respect to $\mu_N$ are uniformly bounded with respect to $N$. Hence the size of this term is essentially the number of terms in the sum. 
To count this number, we observe that once the multindex $(p,j,k,\ell,m)$ is fixed, there is only a finite (and independent of $N$) number of choice for $(p',j',k',\ell',m')$ fulfilling  the orthogonality condition $-1_{p'} + 1_{j'} + 1_{k'} + 1_{\ell'} = -1_{p} + 1_j + 1_k + 1_\ell$.  
Hence the term in the sum is bounded by a universal constant times the cardinal of $\mathcal{A}_K^{N,h}$ which is of order $\mathcal{O}(h^2N^8)$. Hence we have
\begin{equation}
\label{bjl2}
| \eqref{h1}| + |\eqref{h2}| \leq C \frac{t \varepsilon^3 }{h \delta^2}, 
\end{equation}
contributing to the last term in \eqref{eq:rest}. 

Now companion to the integral terms \eqref{h1} and \eqref{h2}, the integration by parts yields boundary terms of the form 
\begin{equation}
\label{h1b}
\frac{\varepsilon^2 N^\er}{h^2 N^4} \sum_{(p,j,k,\ell) \in \mathcal{B}_K^{N,h}} W_{pjk\ell}(\delta)   \langle \widebar V_{p}
V_j V_k V_\ell \rangle (s) \quad\mbox{and}\quad 
\frac{\varepsilon^2 N^\er}{h^2 N^4}\sum_{(p,j,k,\ell) \in \mathcal{B}_K^{N,h}} W_{pjk\ell}(\delta)   \langle \widebar V_{p}
\widebar V_j V_k V_\ell \rangle(s) ,
\end{equation}
where $ |W_{pjk\ell}(\delta)| \leq C \delta^{-2}$, and where $\mathcal{B}_K^{N,h}$ are sets of multindices of cardinal $|\mathcal{B}_K^{N,h}|= \mathcal{O}(h^2N^6)$. Moreover in the second term the sum is taken over indices such that $1_p + 1_j - 1_j - 1_\ell \neq 0$ -the other having contributed to the ``kinetic" part of \eqref{prezakh}-  so that there is no term depending only on the $|V_k|^2$, and hence the term is of zero average (which is obvious for the first term). By using similar techniques as before, we can bound these terms by 
$$
\frac{C \varepsilon^2 N^\er}{h^2 N^4 \delta^2} \left(  \frac{C_0}{N^{2\alpha - 2}}  \right)^{1/2} \left( C h^2 N^6   \right)^{1/2} \leq C \frac{\varepsilon^2 }{h \delta^2}, 
$$
contributing to the second term in \eqref{eq:rest}. 

Now we repeat the procedure for the second term in the right-hand side of \eqref{init}. We find 
\begin{multline*}
\int_0^t  \mathrm{Re} \, 2 \frac{i\varepsilon N^\er}{h^2 N^3} \sum_{\substack{m = -k + \ell \\ m \in \mathcal{C}_K^{N,h}} } \Psi_{mk\ell}^+ e^{i s (-\omega_k + \omega_\ell - \omega_m)  - r_{mk\ell} s \delta}  \langle \widebar V_m \widebar V_k V_\ell \rangle_*(s) \dd s \\
= - 2\int_0^t  \mathrm{Re} \, \frac{\varepsilon N^\er }{h^2 N^3}  \sum_{\substack{m = -k + \ell \\ m \in \mathcal{C}_K^{N,h}} } \Psi_{mk\ell}^+ \frac{e^{i s (-\omega_k + \omega_\ell - \omega_m)  - r_{mk\ell} s \delta }}{ (-\omega_k + \omega_\ell - \omega_m)  +  i r_{mk\ell}\delta}  \frac{\dd}{\dd t}  \langle \widebar V_m \bar V_k V_\ell \rangle_*(s) \dd s\\
+ 2\left[ 
 \mathrm{Re} \, \frac{\varepsilon N^\er}{h^2 N^3} \sum_{\substack{m = -k + \ell \\ m \in \mathcal{C}_K^{N,h}}  } \Psi_{mk\ell}^+ \frac{e^{i s (-\omega_k + \omega_\ell - \omega_m)  - r_{mk\ell} s \delta }}{ (-\omega_k + \omega_\ell - \omega_m)  +  i r_{mk\ell}\delta}   \langle \widebar V_m \widebar V_k V_\ell \rangle_*(s) 
 \right]_0^t. 
\end{multline*}
The boundary term can be estimated as for Equation \eqref{beurk}. For the integral term, we calculate as in \eqref{beurk2} and we isolate only the term symmetric in $(V,\bar V)$ (the other terms being estimated by integration by part as previously), and we obtain 
\begin{align*}
&- 2\int_0^t  \mathrm{Re} \,  \frac{\varepsilon N^\er}{h^2 N^3}  \sum_{\substack{m = -k + \ell\\ m \in \mathcal{C}_K^{N,h}} } \Psi_{mk\ell}^+ \frac{e^{i s (-\omega_k + \omega_\ell - \omega_m)  - r_{mk\ell} s \delta}}{ (-\omega_k + \omega_\ell - \omega_m)  +  ir_{mk\ell} \delta}  \frac{\dd}{\dd t}  \langle \widebar V_m \widebar V_k V_\ell \rangle_*(s) \dd s\\
& = 4 \int_0^t  \mathrm{Re}\,  i\frac{\varepsilon^2 N^\er}{h^2 N^4} \sum_{\substack{m = -k + \ell  \\ m = -j +p \\ m \in \mathcal{C}_K^{N,h}\\ | 1_p - 1_{j} + 1_k - 1_\ell|^2 = r}  
}\Psi_{mk\ell}^+ \Psi_{m j p}^+ \frac{e^{i s (-\omega_k + \omega_\ell + \omega_j - \omega_p)  - r s \lambda}}{ (-\omega_k + \omega_\ell - \omega_{m})  +  ir_{mk\ell} \delta}  \langle  V_j \widebar V_p  \widebar V_k V_\ell \rangle_*(s) \dd s\\
&+ 4\int_0^t  \mathrm{Re}\,  i \frac{\varepsilon^2 N^\er }{h^2 N^4} \sum_{\substack{m = -k + \ell \\ k = -j + p \\  m \in \mathcal{C}_K^{N,h} \\  | -1_m + 1_{j} - 1_p + 1_\ell|^2 = r}} \Psi_{mk\ell}^+ \Psi_{k j p}^+\frac{e^{i s (\omega_j - \omega_p + \omega_\ell - \omega_m)  - r s \delta}}{ (-\omega_{k} + \omega_\ell - \omega_m)  +  i r_{mk\ell}\delta}    \langle \widebar V_m V_j \widebar V_p V_\ell \rangle_*(s) \dd s\\
&- 2\int_0^t  \mathrm{Re} \, i \frac{\varepsilon^2 N^\er}{h^2 N^4} \sum_{\substack{m = -k + \ell \\ \ell = j + p \\ m \in \mathcal{C}_K^{N,h}\\ | -1_m + 1_{j} + 1_p - 1_k|^2 = r}} \Psi_{mk\ell}^+\Psi_{\ell j p}^+ \frac{e^{i s (-\omega_k + \omega_j + \omega_p - \omega_m)  - r s \delta}}{ (-\omega_k + \omega_{\ell} - \omega_m)  +  ir_{mk\ell} \delta}    \langle \widebar V_m \widebar V_k V_j V_p \rangle_*(s) \dd s\\
&+ \mathcal{O}( \frac{\varepsilon^2 }{h\delta^2} + \frac{t \varepsilon^3}{h\delta^2}). 
\end{align*}
The terms for which $r = 0$ are equal to 
\begin{align*}
&4 \int_0^t  \mathrm{Re} \, i \frac{\varepsilon^2 N^\er}{h^2 N^4}  \sum_{\substack{m = -p + j \\ m \in \mathcal{C}_K^{N,h}}}  (\Psi_{mjp}^+)^2 \frac{1}{ (-\omega_j + \omega_p - \omega_{m})  + i r_{mjp}  \delta}  \langle | V_j V_p |^2 \rangle(s) \dd s\\
&+ 4\int_0^t  \mathrm{Re} \,  i \frac{\varepsilon^2 N^\er}{h^2 N^4}  \sum_{\substack{m = -j + p  \\ m \in \mathcal{C}_K^{N,h}} } (\Psi_{m j p}^+)^2 \frac{1}{ (-\omega_{j} + \omega_p - \omega_m)  +  i r_{mjp} \delta}    \langle | V_m V_p |^2 \rangle(s) \dd s\\
&- 4\int_0^t  \mathrm{Re} \, i \frac{\varepsilon^2 N^\er}{h^2 N^4} \sum_{\substack{m = -j + p \\ m \in \mathcal{C}_K^{N,h}}} (\Psi_{mjp}^+)^2 \frac{1}{ (-\omega_j + \omega_{p} - \omega_m)  + r_{mjp} i \delta}  \langle | V_m V_j|^2 \rangle(s) \dd s. 
\end{align*}
These terms give the second contribution in \eqref{prezakh} up to estimate of the form \eqref{eqnlinearization}. 
The other term for which $r > 1$ can be integrated by part again and estimated as previously. 
\end{proof}
We can write the previous result in the original formulation with frequencies in $\mathcal{D}_N$ and not only in $\mathcal{D}^+_N$. Recall that $V_j$ can be extended for $j \in \mathcal{D}$ by the relation $\widebar V_{j}(t) = V_{-j}(t)$. The fluctuation terms $F_m^N(t)$ defined by \eqref{eqFm} in $\mathcal{D}_N$ satisfy $F_{-m}^N(t) = F_m^N(t)$, and a similar relation holds for the coarse-grained fluctuations $F_K^{N,h}(t) = F_{-K}^{N,h}(t)$. The previous result can be written 
\begin{proposition} Under Hypothesis \ref{hypinit}, there exist constants $C$, $N_0$, $h_0$ and $\delta_0$  such that for all $N\geq N_0$, $h \leq h_0$, $\delta \leq \delta_0$ and for all $t \geq 0$ and all $K \in \mathcal{G}_h$,   
\begin{multline}
\label{prezakh2}
F_K^{N,h}(t) = F_K^{N,h}(0) + 2  \int_0^t  \frac{\varepsilon^2 N^\er}{h^2 N^4}  \times \\
\sum_{\substack{ m = j+p \\ m \in \mathcal{C}_K^{N,h}} }  
 \frac{3 \delta (\Psi_{mj p})^2\Big(  \langle | V_j |^2 | V_p|^2  \rangle - \mathrm{sign}(m_x j_x) \langle | V_m|^2 | V_j|^2  \rangle - \mathrm{sign}(m_x p_x) \langle | V_m|^2 | V_p|^2  \rangle\Big)(s)}{ (\omega_m - \omega_j - \omega_{p})^2  + 3^2 \delta^2 }  \dd s\\
+  R(\varepsilon, K,N,h,\delta,t) 
\end{multline}
with $R(\varepsilon, K,N,h,\delta,t)$ satisfying \eqref{eq:rest}. 

\end{proposition}

\section{Linearization and discrete resonant concentration\label{linearization}}
The next result shows the linearization of the previous equation: 
\begin{proposition}
Under the previous hypothesis, we have for all $N\geq N_0$, $t \geq 0$, $K \in \mathcal{G}_h$, $h \leq h_0$ and $\delta \leq \delta_0$,   

\begin{equation}
\label{linzakhregN}
\begin{split}
F_K^{N,h}(t) = & F_K^{N,h}(0) +  \frac{ \varepsilon^2}{ h^2 N^4} \int_0^t \sum_{\substack{m = j + p \\ m \in \mathcal{C}_K^{N,h}} } \frac{3\delta }{( \omega_m - \omega_{j} - \omega_p)^2 + 3^3\delta^2}    L(m,j,p) F_m^N(s) \\
&  + \frac{ \varepsilon^2}{ h^2 N^4}  \int_0^t \sum_{\substack{m = j + p \\ m \in \mathcal{C}_K^{N,h}} } \frac{3 \delta }{( \omega_m - \omega_{j} - \omega_p)^2 +  3^3\delta^2}   S(m,j,p) F_p^N(s)   \\
& + \frac{ \varepsilon^2}{ h^2 N^4} \int_0^t \sum_{\substack{m = j + p \\ m \in \mathcal{C}_K^{N,h}}  } \frac{3 \delta }{( \omega_m - \omega_{j} - \omega_p)^2 +  3^3\delta^2}  S(m,p,j) F_j^N(s)  
\\
&+ R_2(\varepsilon, K,N,h,\delta,t)
\end{split}
\end{equation}
with $L$ and $S$ defined in \eqref{eq:LS}, and 
\begin{equation}
\label{eq:rest2}
\sup_{K \in \mathcal{G}_h} |R_2(\varepsilon, K,N,h,\delta,t)| \leq C \Big( \frac{ \varepsilon }{ h \delta} + \frac{\varepsilon^2 }{ h \delta^2}+ \frac{t \varepsilon^3 }{ h \delta^2} 
+ \frac{ t \varepsilon^2}{\delta h N} +  \frac{t \varepsilon^2 \delta}{N^{2-\alpha}} \Big)
\end{equation}
\end{proposition}

\begin{proof}
In Equation \eqref{prezakh2}, all the terms are quadratic in $|V_j|^2$. We can write all of them as 
$$
\langle | V_j |^2 | V_k|^2  \rangle = \gamma_j \gamma_k + \gamma_k (\langle |V_j|^2 \rangle - \gamma_j ) + \gamma_j (\langle |V_k|^2 \rangle - \gamma_k) + \langle(|V_j|^2 - \gamma_j ) (|V_k|^2 - \gamma_k)\rangle. 
$$
As $\gamma_j$ is a stationnary state of the kinetic equation cancelling the integrand in the equation \eqref{zakh}, the contribution coming from the terms in $\gamma_j \gamma_k$ vanishes. The contribution of the second type of terms of the form $\gamma_k (\langle |V_j|^2 \rangle - \gamma_j )$ gives the main terms in \eqref{linzakhregN} after a rescaling in $N^\alpha$. So we need to estimate the terms coming from the last contribution, quadratic in  $( |V_k|^2  - \gamma_k)$. Three terms of similar form are to be estimated. The first one is 
\begin{equation}
\label{dijon}
\frac{\varepsilon^2 N^\er}{h^2N^4}
\sum_{\substack{m = j+k \\ m \in \mathcal{C}_K^{N,h}}}  
 \frac{3 \delta (\Psi_{mj k})^2}{ (\omega_m - \omega_j - \omega_{k})^2  + 3^2 \delta^2 }  \langle (| V_j |^2 - \gamma_j) ( | V_k|^2 - \gamma_k)  \rangle. 
\end{equation}
Note that in this sum, we cannot have $ j = -k$ which implies $ m = 0$ and the corresponding term is zero. 
When $j = k$, we have $j = k = \frac{m}{2}$. But such terms are non resonant as shown in  \eqref{orage1} and we can use a bound of the form \eqref{orage2}. As we have 
$$
\Big|\langle (| V_{\frac{m}{2}} |^2 - \gamma_\frac{m}{2})^2 \rangle \Big| \leq 
\left(\int  \frac{\rho_N^2} {\mu_N}\right)^{\frac12} \left(
 \int  ( | v_{\frac{m}{2}} |^2 - \gamma_\frac{m}{2})^4  \mu_N \right)^{\frac12} \leq C
$$ 
by using \eqref{fluct2} with $\er \geq 1$ and \eqref{eq:mkp}. 
We deduce that (compare \eqref{eqnlinearization})
$$
\left| \frac{\varepsilon^2 N^\er}{h^2N^4}
\sum_{ m \in \mathcal{C}_K^{N,h}}  
 \frac{3 \delta (\Psi_{m \frac{m}{2} \frac{m}{2} })^2}{ (\omega_m - 2\omega_\frac{m}{2} )^2  + 3^2 \delta^2 }  \langle (| V_{\frac{m}{2}} |^2 - \gamma_\frac{m}{2})^2 \rangle 
 \right| \leq \frac{C \delta \varepsilon^2 N^\er}{h^2 N^4} \left(\sum_{ m \in \mathcal{C}_K^{N,h}} 1 \right)  \leq C \frac{\varepsilon^2 \delta}{N^{2 - \alpha}}, 
$$
and this term contributes to the  last term in \eqref{eq:rest2}. 

Hence in the sum \eqref{dijon}, we can exclude the terms for which $j = \pm k$, and  by orthogonality, we can write
$$
\langle (| V_j |^2 - \gamma_j) ( | V_k|^2 - \gamma_k)  \rangle  = \int (| v_j |^2 - \gamma_j) ( | v_k|^2 - \gamma_k) ( \rho_N(t) - \mu_N)
$$
and by estimating the sum as before, 
we obtain that it is  bounded by 
$$
\frac{\varepsilon^2 N^\er}{ \delta h^2 N^{4}} \times \Big(\frac{C_0}{N^{2\er - 2}}\Big)^{\frac12} \times \big( C h^2 N^4)^{\frac12}  = \frac{C}{\delta h N}, 
$$
yielding the penultimate contribution in \eqref{eq:rest2}. 
The other terms are estimated similarly. 
\end{proof}

\underline{End of the proof of Theorem \ref{th1}.} 

Note that for $t \leq T/(\pi \varepsilon^2)$ with the notations of the Theorem, we have using \eqref{eq:rest2} that 
$$
\sup_{\substack{K \in \mathcal{G}_h \\ t \leq T/(\pi \varepsilon^2)}} |R_2(\varepsilon, K,N,h,\delta,t)| \leq C \Big(  \frac{ \varepsilon }{ h \delta^2} 
+ \frac{ 1}{\delta h N} +  \frac{\delta}{N^{2-\alpha}} \Big), 
$$
provided that $\varepsilon$ and $\delta$ are small enough. 

In the remainder of the proof, we set $\lambda =  3 \delta$. 

Let $f_\lambda(t,m) \in C^1(\mathcal{D})$ be the solution of the linearized kinetic equation \eqref{linzakh3} on the interval $[0,T]$ with initial data $f_\lambda(0,m) = g_0(m)$, $m \in \mathcal{D}$. As noticed in Remark \ref{rkderivees}, if we have $\Norm{f_\lambda}{L^\infty(\mathcal{D})} \leq C$ uniformly in $\lambda$, we have only $\Norm{f_\lambda}{W^{1,\infty}(\mathcal{D})} \leq C \lambda^{-2}$ (see \eqref{mamachou}). 

We define $\widetilde f_\lambda (t,m) = f_\lambda(\pi \varepsilon^2 t,m)$ and for $m \in \mathcal{D}_N$ and $K \in \mathcal{G}_h$, 
$$
r_m^N(s) = F_m^N(s) - \widetilde f_\lambda(s,m) \quad \mbox{and}\quad r_K^{N,h}(s) = \frac{1}{h^2 N^2} \sum_{m \in \mathcal{C}_K^{N,h}}r^N_m(s). 
$$
Note that we have $r_K^{N,h}(0) = 0$. 
By definition, $\widetilde f_{\lambda}(t,m) = f_\lambda(\pi \varepsilon^2 t,m)$ satisfies 
$$
\widetilde f_{\lambda}(t,m) = g_0(m) +  \pi \varepsilon^2 \int_0^t (\mathcal{L_\lambda} \widetilde f_\lambda)(s,m) \dd s, 
$$
for $t \leq T/(\pi\varepsilon^2)$, where $\mathcal{L}_\lambda$ is the operator defining the quasi-resonant equation \eqref{linzakhreg}. 
Now, using the estimates \eqref{traindijon} and  \eqref{mamachou},  we have that 
\begin{multline*}
\pi (\mathcal{L_\lambda} \widetilde f_\lambda)(s,m) =  \frac{ 1}{ N^2} \sum_{m = j + p} \frac{\lambda }{( \omega_m - \omega_{j} - \omega_p)^2 + \lambda^2}     L(m,j,p)\widetilde f_\lambda(s,m) \\
 \frac{ 1}{ N^2} \sum_{m = j + p} \frac{\lambda }{( \omega_m - \omega_{j} - \omega_p)^2 + \lambda^2}     (Q(m,j,p) \widetilde f_\lambda(s,p) +  Q(m,p,j)\widetilde f_\lambda(s,j)) + \mathcal{O}(\frac{1}{N\lambda^2}), 
\end{multline*}
where the last term comes from standard estimates between continuous and discrete integral of mesh $1/N$, owing to the fact that the integrand has a derivative uniformly bounded by $1/\lambda^{2}$. 

Hence, we obtain in view of \eqref{linzakhregN}
\begin{equation}
\label{linzakhregN2}
\begin{split}
r_K^{N,h}(t) = & r_K^{N,h}(0) +  \frac{ \varepsilon^2}{ h^2 N^4} \int_0^t \sum_{\substack{m = j + p \\ m \in \mathcal{C}_K^{N,h} }} \frac{\lambda }{( \omega_m - \omega_{j} - \omega_p)^2 + \lambda^2}    L(m,j,p) r_m^N(s)  \\
&  + \frac{ \varepsilon^2}{ h^2 N^4}  \int_0^t \sum_{\substack{m = j + p \\ m \in \mathcal{C}_K^{N,h} } } \frac{\lambda }{( \omega_m - \omega_{j} - \omega_p)^2 + \lambda^2}   Q(m,j,p) r_p^N(s)  \\
& + \frac{ \varepsilon^2}{ h^2 N^4} \int_0^t \sum_{\substack{m = j + p \\ m \in \mathcal{C}_K^{N,h} } } \frac{\lambda }{( \omega_m - \omega_{j} - \omega_p)^2 + \lambda^2}  Q(m,p,j) r_j^N(s)  
\\
&+ R_3(\varepsilon, K,N,h,\delta,t)
\end{split}
\end{equation}
with for $t \leq T/(\pi\varepsilon^2)$ and $\varepsilon_0$ and $\delta_0$ small enough, 
\begin{equation}
\label{eq:rest3}
\sup_{\substack{K \in \mathcal{G}_h \\ t \leq T/(\pi\varepsilon^2)} } | R_3(\varepsilon, K,N,h,\delta,t)| \leq C \Big( \frac{ \varepsilon}{h \delta^2} 
+ \frac{ 1}{h \delta N}  + \frac{1}{\delta^2 N} + \frac{\delta}{N^{2 - \er}}\Big). 
\end{equation}
Now the first term in the time integral of the right-hand side of \eqref{linzakhregN2} can be written 
$$
 \frac{ \varepsilon^2}{ h^2 N^2} \sum_{ m \in \mathcal{C}_K^{N,h} } \left( \frac{1}{N^2}\sum_{j \in \mathcal{D}_N} \frac{\lambda }{( \omega_m - \omega_{j} - \omega_{m-j})^2 + \lambda^2}    L(m,j,m-j) \right) r_m^N(s)
$$
and we have by standard approximation on the subgrid $\mathcal{G}_h$ that for all $m \in \mathcal{C}_K^{N,h}$, 
\begin{multline*}
 \frac{1}{N^2}\sum_{j\in \mathcal{D}_N} \frac{\lambda }{( \omega_m - \omega_{j} - \omega_{m-j})^2 + \lambda^2}    L(m,j,m-j)  \\=  h^2\sum_{J\in \mathcal{G}_h}  \frac{1}{h^2 N^2} \sum_{j \in \mathcal{C}_J^{N,h}} \frac{\lambda }{( \omega_m - \omega_{j} - \omega_{m-j})^2 + \lambda^2}    L(m,j,m-j) 
\\ = h^2\sum_{J \in \mathcal{G}_h} \frac{\lambda }{( \omega_K - \omega_{J} - \omega_{K-J})^2 + \lambda^2}    L(K,J,K-J) + \mathcal{O}( \frac{h}{\lambda^2}), 
\end{multline*}
where the last term comes from the estimate \eqref{traindijon} of the derivative of the summand. Hence we obtain 
\begin{multline*}
 \frac{ 1}{ h^2 N^4} \sum_{\substack{m = j + p \\ m \in \mathcal{C}_K^{N,h} }} \frac{\lambda }{( \omega_m - \omega_{j} - \omega_p)^2 + \lambda^2}    L(m,j,p) r_m^N(s) \\
 = h^2\sum_{J \in \mathcal{G}_h} \frac{\lambda }{( \omega_K - \omega_{J} - \omega_{K-J})^2 + \lambda^2}    L(K,J,K-J) r_K^{N,h} (s)+ \mathcal{O}\big( \frac{h}{\delta^2}\sup_{K} | r_K^{N,h}(s)|\big). 
\end{multline*}
The second and third terms in Equation \eqref{linzakhregN2} can be written equivalently
$$
h^2 \sum_{J \in \mathcal{G}_h} \frac{1}{h^2 N^2}\sum_{ j \in \mathcal{C}_K^{N,h} }r_j^N(s) \left( \frac{ 1}{ h^2 N^2}  \sum_{m \in \mathcal{C}_K^{N,h}  } \frac{\lambda }{( \omega_m - \omega_{m-j} - \omega_j)^2 + \lambda^2}   Q(m,m-j,j) \right) 
$$
and as before, we have for $j \in \mathcal{C}_J^{N,h}$, 
\begin{multline*}
\frac{ 1}{ h^2 N^2}  \sum_{m \in \mathcal{C}_K^{N,h}  } \frac{\lambda }{( \omega_m - \omega_{m-j} - \omega_j)^2 + \lambda^2}   Q(m,m-j,j)  \\
= \frac{\lambda }{( \omega_K - \omega_{K-J} - \omega_J)^2 + \lambda^2}   Q(K,K-J,J) + \mathcal{O}(\frac{h}{\delta^2}). 
\end{multline*}
We therefore obtain a similar estimate as for the first term. Eventually, we have that
\begin{equation}
\label{linzakhregN3}
\begin{split}
r_K^{N,h}(t) = & r_K^{N,h}(0) +  \varepsilon^2 \int_0^t h^2 \sum_{K = J + P} \frac{\lambda }{( \omega_K - \omega_{J} - \omega_P)^2 + \lambda^2}    L(K,J,P) r_K^{N,h}(s) \dd s  \\
&  +  \varepsilon^2  \int_0^t h^2 \sum_{K = J + P} \frac{\lambda }{( \omega_K - \omega_{J} - \omega_P)^2 + \lambda^2}  Q(K,J,P) r_P^{N,h}(s) \dd s \\
& + \varepsilon^2\int_0^t h^2 \sum_{K = J + P} \frac{\lambda }{( \omega_K - \omega_{J} - \omega_P)^2 + \lambda^2}   Q(K,P,J) r_J^{N,h}(s) \dd s 
\\
&+ R_3(\varepsilon, K,N,h,\delta,t) + R_4(\varepsilon, K,N,h,\delta,t, r_K)
\end{split}
\end{equation}
with $R_3$ satisfying \eqref{eq:rest3}, and 
$$
\sup_{K \in \mathcal{G}_h} |R_4(\varepsilon, K,N,h,\delta,t, r_K)| \leq C \frac{h}{\delta^2} \varepsilon^2 \int_0^t \sup_{K} | r_K^{N,h}(s)| \dd s. 
$$
Let us set  $\Norm{r^{N,h}(s)}{\ell^\infty} := \sup_{K \in \mathcal{G}_h} |r_K^{N,h}(s)|$. 
From the previous equation, we obtain that 
$$
\Norm{r^{N,h}(t)}{\ell^\infty} \leq \Norm{r^{N,h}(0)}{\ell^\infty} + \sigma_0  + C \varepsilon^2 \int_0^t ( \Norm{\mathcal{L}_\lambda^h}{} + \frac{h}{\delta^2}) \Norm{r^{N,h}(s)}{\ell^\infty}  \dd s 
$$
where $\sigma_0 = \sup_{K,t \leq T/(\pi\varepsilon^2)} |R_3(\varepsilon, K,N,h,\delta,t)|$ and 
$$
\Norm{\mathcal{L}_\lambda^h}{} = \sup_{K} \left|
h^2 \sum_{K = J + P} \frac{\lambda }{( \omega_K - \omega_{J} - \omega_P)^2 + \lambda^2}    ( L(K,J,P)  + Q(K,J,P) + Q(K,P,J)) 
\right|
$$
Now we see that the right-hand side of this equation is a discretisation of the upper bound of $\Norm{\mathcal{L}_\lambda}{L^\infty}$ with mesh $h$. As this term is uniformly bounded when $\lambda \to 0$ and as the integrand is smooth with derivative bounded by $C\lambda^{-2}$, we deduce that (with $\lambda = 3 \delta$) 
$$
\Norm{\mathcal{L}_\lambda^h}{} \leq C( 1 + \frac{h}{\lambda^2}). 
$$
We deduce by using the Gr\"onwall lemma that for all $t \leq T/(\pi\varepsilon^2)$, when $h \leq \delta^2$, we have 
$$
\Norm{r^{N,h}(t)}{\ell^\infty} \leq C (\sigma_0 + \Norm{r^{N,h}(0)}{\ell^\infty}) \leq C \Big( \frac{ \varepsilon}{h \delta^2} 
+ \frac{ 1}{h \delta N} + \frac{\delta}{N^{2 - \er}}\Big), 
$$
as 
$$
\frac{1}{\delta^2 N} \leq \frac{ 1}{h \delta N}  
$$
under the assumption $h \leq \delta^2$, and if $\delta$ is small enough. 
We thus easily obtain the result. 

\begin{remark}
The proof of Corollary \ref{cor1} is obvious after noticing that for $K \in \mathcal{G}_h$ and $t \leq T$, 
\begin{eqnarray*}
\left(\frac{1}{h^2N^2}\sum_{k \in \mathcal{C}_K^{N,h}} f_\lambda(t,k)\right)  - f(t,K) &=& 
\frac{1}{h^2N^2}\sum_{k \in \mathcal{C}_K^{N,h}} (f_\lambda(t,k) - f(t,k)) \\
&& + \frac{1}{h^2N^2}\sum_{k \in \mathcal{C}_K^{N,h}} (f(t,k) - f(t,K)),
\end{eqnarray*}
since $\mathrm{Card} \, \mathcal{C}_K^{N,h} = h^2 N^2$. In the right-hand side of this equation, the first term is bounded by $C \sqrt{\lambda}$ with $\lambda = 3 \delta$ by using \eqref{th3conv}, and the second term is bounded by $C h \leq C \delta^2$ by using the regularity of $f(t,k)$ given by Theorem \ref{th3}. Note also that using \eqref{mamachou}, we easily obtain the estimate 
$$
\frac{1}{h^2N^2}\sum_{k \in \mathcal{C}_K^{N,h}} f_\lambda(t,k) = f_\lambda(t,K) + \mathcal{O}(\frac{h}{\lambda^2}). 
$$
\end{remark}

\section{The deterministic equation with random initial conditions}

We now give  the proof of Theorem \ref{th2}. Let us first remark that Proposition \ref{prop2} is valid for all $\delta \geq 0$, and that in fact for $\delta = 0$, \eqref{fluct2} is an equality. Hence all the bounds in the proof for the stochastic case will be valid, up to estimates on the small denominators. These diophantine controls are given by the following estimates: 
\begin{lemma}
Recall that for $k = (k_x,k_y) \in \mathcal{D}$, $\omega_k = k_x^3 + \eta k_y^2 k_x^{-1}$. 
There exists $\nu > 0$ such that 
for almost all $\eta >0$,  there exists a constant $c$ such that we have for all $N \in \N$, 
\begin{equation}
\label{nonres}
\forall\, (k,j,\ell) \in \mathcal{D}_N^+\qquad 
|\omega_{k+j} - \omega_k - \omega_j| > \frac{c}{N^\nu}\quad \mbox{and}
\quad |\omega_{k+j+\ell} - \omega_k - \omega_j - \omega_\ell| > \frac{c}{N^\nu}. 
\end{equation}
Moreover, defining for a given $m \in \mathcal{D}_N^+$ the resonant set 
$$
\mathcal{R}_m^N = \{ (j,k,\ell) \in (\mathcal{D}_N^+)^3 | \, m = k - j + \ell \quad\mbox{and} \quad \omega_m = \omega_k - \omega_j + \omega_\ell\}
$$
there exists a constant $C$ such that 
for all $m$ almost all $\eta$ 
\begin{equation}
\label{cardR}
\mathrm{Card} \, \mathcal{R}_m^N \leq C N^2
\end{equation}
and 
\begin{equation}
\label{nonres2}
\forall\, (j,k,\ell) \notin \mathcal{R}_m^N, \quad | \omega_m + \omega_j - \omega_k - \omega_\ell| > \frac{c}{N^\nu}.
\end{equation}
\end{lemma}

\begin{proof}
Let $\epsilon > 0$. 
From a  version of Khinchin's Theorem (see for instance \cite{Khinchin}), we know that for almost all $\eta > 0$, there exists $c$ such that 
\begin{equation}
\label{khinchin}
\forall\, (p,q)\in \mathbb{Z} \times \mathbb{Z}^*\quad 
\left|\frac{1}{\eta} - \frac{p}{q}\right| \geq \frac{c}{|q|^{2 + \epsilon}}.
\end{equation}
For $k = (k_x,k_y)$ and $j = (j_x,j_y)$ in $\mathcal{D}_N^+$, we calculate that 
\begin{eqnarray*}
\omega_{k+j} - \omega_k - \omega_j &=& 3 k_x j_x (k_x + j_x) + \eta \Big( \frac{(k_y + j_y)^2}{k_x + j_x} -\frac{k_y^2}{k_x} - \frac{j_y^2}{j_x} \Big)
\\
&=& \eta 3 r \left( \frac{1}{\eta} + \frac{s}{3r^2}\right),
\end{eqnarray*}
where $r = k_x j_x (k_x + j_x)$ and $s = (k_y + j_y)^2 j_x k_x - (k_x + j_x) ( k_y^2 j_x + j_y^2 k_x)$. By definition of $\mathcal{D}^+$, we have $r > a$ for some constant $a$ depending on $\mathcal{D}^+$. Moreover, by definition of the discrete grid, we have 
$s = S N^{-4}$ and $r = RN^{-3}$ for some integers $S$ and $R\neq 0$ with $|S| \leq C N^4$ and $|R| \leq C N^3$ for some constant $C$ depending on $\mathcal{D}$. Hence we have 
$$
|\omega_{k+j} - \omega_k - \omega_j|\geq \eta a \left| \frac{1}{\eta} + \frac{SN^2}{3 R^2}   \right|\geq \frac{c_1}{|R|^{4+ 2\epsilon}}\geq \frac{c_2}{N^{12 + 6\epsilon}}
$$
after using \eqref{khinchin}, and for some constants $c_1$ and $c_2$ depending on $\eta$ and $\mathcal{D}$. This proves the first part of the statement\footnote{This estimate is far from being optimal. We could derive much better estimates by using the exact formulas for the resonant manifold in Lemma \ref{lemma21}.}. 

Similarly, we can write 
$$
\omega_{k+j+\ell} - \omega_k - \omega_j - \omega_\ell = r + \eta \frac{s}{t}
$$
with $r = (k_x + j_x + \ell_x)^3 - k_x^3 - j_x^3 - k_x^3$,  $t = (k_x + j_x + \ell_x) j_x \ell_x k_x$ and $s$ a homogeneous polynomial of order $5$ in $(k_x,k_y,j_x,j_y,\ell_x,\ell_y)$. We have that $r$ and $s$ are positive and bounded from below by some constant $a$. Hence, we have 
$$
|\omega_{k+j+\ell} - \omega_k - \omega_j - \omega_\ell| \geq \eta \left| a \Big(\frac{1}{\eta} + \frac{s}{tr}\Big)\right|. 
$$
Now arguing as before, by homogeneity of the polynomials we have $s = SN^{-5}$, $t = T N^{-4}$ and $r = R N^{-3}$ with $|S|\leq C N^5$, $|T|\leq C N^4$ and $|R|\leq C N^{3}$, and we finally obtain that 
$$
|\omega_{k+j+\ell} - \omega_k - \omega_j - \omega_\ell| \geq c_1\left|\frac{1}{\eta} + \frac{SN^2}{TR}\right| \geq \frac{c_2}{N^{12 + 6\epsilon}}
$$
by using \eqref{khinchin}, and for some constants $c_1$ and $c_2$ depending on $\eta$ and $\mathcal{D}$. This shows \eqref{nonres}. 

To proof of \eqref{nonres2} is entirely similar, but we need to study the resonant modulus $\mathcal{R}_m^N$. 
Let $(m,j,k,\ell) \in \mathcal{R}_m^N$. We have 
$$
\omega_m + \omega_j - \omega_k - \omega_\ell = m_x^3 + j_x^3 - k_x^3 - \ell_x^3 + \eta ( \frac{m_y^2}{m_x} + \frac{j_y^2}{j_x} - \frac{k_y^2}{k_x} - \frac{\ell_y^2}{\ell_x}).  
$$
As we can assume that $\eta \notin \mathbb{Q}$, we must have 
$$
m_x^3 + j_x^3 - k_x^3 - \ell_x^3 = 0, \quad \mbox{and} \quad \frac{m_y^2}{m_x} + \frac{j_y^2}{j_x} - \frac{k_y^2}{k_x} - \frac{\ell_y^2}{\ell_x} = 0. 
$$
Using the condition $m_x + j_x  = k_x + \ell_x$, the first relation implies that  $m_x^2 j_x + m_x j_x ^2 = k_x^2 \ell_x+ k_x \ell_x^2$ and hence as $m_x + j_x > 0$, 
$ m_x j_x = k_x \ell_x$, from which we deduce that $m_x^2 + j_x^2 - k_x^2 - \ell_x^2 = 0$ and finally, that $m_x = k_x$ and $j_x = \ell_x$ or $m_x = j_x$ and $m_x = k_x$. 
Let us assume that we are in the first situation. The second relation implies 
$$
j_x(m_y^2 - k_y^2) =  m_x  (\ell_y^2 - j_y^2). 
$$
Note that $m_y - k_y = \ell_y - j_y$. If this quantity vanishes, we are in the situation where $m = k$ and $j = \ell$ and we obtain indeed $\mathcal{O}(N^2)$ possible terms for a given $m$. 
If $m_y - k_y \neq 0$, we obtain 
$$
j_x(m_y + k_y) =  m_x ( \ell_y + j_y) 
$$
and we can verify that the resonant terms are given by 
$$
\ell_y  = \frac{j_x}{2 m_x} (m_y + k_y)  + \frac{1}{2} (m_y - k_y) \quad\mbox{and}\quad 
j_y =  \frac{j_x}{2 m_x} (m_y + k_y) - \frac{1}{2}(m_y - k_y). 
$$
and we see that there is only one degree of freedom ($k_y$) in the $y$-variable and one in the $x$ variable, which shows that the cardinal of the resonant modulus is of size $\mathcal{O}(N^2)$. Note that the exact resonant terms depend on arithmetic conditions on the indices to ensure that $\ell_y$ and $j_y$ indeed belong to the lattice $\mathcal{D}_N^+$. 
\end{proof}
%

To prove Theorem \ref{th2}, 
we follow the lines of the proof of Theorem \ref{th1} with $\delta = 0$ and using the bounds of the previous Lemma instead of \eqref{bounddelta}. Starting from \eqref{init} to estimate $F_K^{N,h}(t) - F_K^{N,h}(0)$ with $\delta = 0$, we can perform the same integration by part as in \eqref{beurk} by using now the bound (see \eqref{nonres}) 
\begin{equation}
\label{bounddelta2}
\left| \frac{1}{\omega_k + \omega_\ell - \omega_m} \right|  \leq C N^{\nu}, 
\end{equation}
so that the first boundary term can be estimated by 
\begin{equation}
\label{boundterm2}
\left|\frac{\varepsilon N^\er}{h^2 N^3}  \sum_{\substack{m = k + \ell\\ m \in \mathcal{C}_K^{N,h}} } \Psi_{mk\ell}^+ \frac{\langle \widebar V_m V_k V_\ell \rangle(s)}{ (\omega_k + \omega_\ell - \omega_m)  }    \right| \leq \frac{C\varepsilon N^\nu}{ h }
\end{equation}
which is essentially the same as \eqref{boundterm} with $\delta = 1/N^\nu$. Now as in \eqref{beurk2}, we compute the contribution of the time integral and find 
\begin{align}
\label{beurk2reson}
&- \int_0^t  \mathrm{Re} \,  \frac{\varepsilon N^\er}{h^2 N^3}  \sum_{\substack{m = k + \ell\\ m \in \mathcal{C}_K^{N,h}} } \Psi_{mk\ell}^+ \frac{e^{i s (\omega_k + \omega_\ell - \omega_m)  }}{ (\omega_k + \omega_\ell - \omega_m) }  \frac{\dd}{\dd t}  \langle \widebar V_m V_k V_\ell \rangle_*(s) \dd s\\
&\nonumber =  \int_0^t  \mathrm{Re}\,   \frac{i\varepsilon^2 N^\er}{h^2 N^4}  \sum_{\substack{m = k + \ell \\ m = j + p \\m \in \mathcal{C}_K^{N,h}}  
}\Psi_{mk\ell}^+ \Psi_{m j p}^+\frac{e^{i s (\omega_k + \omega_\ell - \omega_j - \omega_p) }}{ (\omega_k + \omega_\ell - \omega_{m})  }  \langle \widebar V_j \widebar V_p  V_k V_\ell \rangle_*(s) \dd s\\
& \nonumber - 2 \int_0^t  \mathrm{Re} \, \frac{i\varepsilon^2 N^\er}{h^2 N^4}  \sum_{\substack{m = k  + \ell \\ k = j + p \\ m \in \mathcal{C}_K^{N,h}}} \Psi_{m k\ell}^+ \Psi_{k jp}^+ \frac{e^{i s (\omega_p + \omega_j + \omega_\ell - \omega_m)  }}{ (\omega_{k} + \omega_\ell - \omega_m)  }    \langle \widebar V_m V_j V_p  V_\ell \rangle_*(s) \dd s \\
& \nonumber + 2 \int_0^t  \mathrm{Re} \,\frac{i\varepsilon^2 N^\er}{h^2 N^4} \sum_{\substack{m = k + \ell \\ m = -j + p \\m \in \mathcal{C}_K^{N,h}}  
}\Psi_{mk\ell}^+ \Psi_{m j p}^+ \frac{e^{i s (\omega_k + \omega_\ell + \omega_j - \omega_p) }}{ (\omega_k + \omega_\ell - \omega_{m})  }  \langle  V_j \widebar V_p  V_k V_\ell \rangle_*(s) \dd s\\
&  \nonumber- 4 \int_0^t  \mathrm{Re}\,  \frac{i\varepsilon^2 N^\er}{h^2 N^4}  \sum_{\substack{m = k  + \ell \\ k = -j + p \\ m \in \mathcal{C}_K^{N,h}}} \Psi_{mk\ell}^+ \Psi_{k j p}^+\frac{e^{i s (\omega_p - \omega_j + \omega_\ell - \omega_m)  }}{ (\omega_{k} + \omega_\ell - \omega_m) }    \langle \widebar V_m \widebar V_j V_p  V_\ell \rangle_*(s) \dd s. 
\end{align}
By using the previous lemma, we see that we can perform an integration by part in as in the stochastic term, for all non resonant terms (the one cancelling the difference of frequencies in the exponentials). By using \eqref{nonres} and \eqref{nonres2}, we see that we will obtain essentially the same estimates as in the stochastic case, by replacing $\delta$ by $N^{-\nu}$. It remains to study the contribution of the resonant terms. 

In view of \eqref{nonres}, we observe that they are no resonant monomials in the the second and third terms contributing to the previous equation. 

For the first term, the term depending on the resonant monomials  can be written  
$$
\mathrm{Re}\,  i\frac{\varepsilon^2 N^\er}{h^2N^4}  \sum_{\substack{m = k + \ell \\ m = j + p \\ m \in \mathcal{C}_K^{N,h} \\ \omega_k + \omega_\ell = \omega_j + \omega_p }  
}\Psi_{mk\ell}^+ \Psi_{m j p}^+\frac{\langle \widebar V_j \widebar V_p  V_k V_\ell \rangle(s)}{ (\omega_k + \omega_\ell - \omega_{m})  }   \dd s. $$
As  $
\overline{\langle \widebar V_j \widebar V_p  V_k V_\ell \rangle} = \langle  V_j  V_p \widebar V_k \widebar V_\ell \rangle
$, we can exchange the indices $(j,p)$ and $(k,\ell)$ and by using the resonance relation, we can prove that the sum is real. Hence by symmetry the term in the previous equation vanishes.

Such a symmetry does not hold for the fourth term, which is written
$$
 \mathrm{Re}\,  \frac{i\varepsilon^2 N^\er}{h^2 N^4}  \sum_{\substack{m = k  + \ell \\ k = -j + p \\ m \in \mathcal{C}_K^{N,h} \\ \omega_p - \omega_j + \omega_\ell - \omega_m = 0}} \Psi_{mk\ell}^+ \Psi_{k j p}^+\frac{ \langle \widebar V_m \widebar V_j V_p  V_\ell \rangle(s)   }{ (\omega_{k} + \omega_\ell - \omega_m) }   \dd s. 
$$
but for this one, we can use \eqref{cardR}. 
Two type of resonant terms are in this sum: first the ones for which $m = p$ and $j = \ell$ or $m = \ell$, $j = p$. These terms depend only on $|V_m|^2|V_p|^2$ and for them the sum is real. Hence these term give no contribution. For the other type of terms, the indices $m$, $ j$, $p$ and $\ell$ must be all different. By estimating as in \eqref{firstest}, using the fact that the number of resonant terms is bounded by $\mathcal{O}(N^2)$ for a given $m$, we obtain that this term is bounded by 
$$
\frac{C \varepsilon^2 N^\er}{h^2 N^4} \left( \frac{1}{N^{2\alpha - 2}}\right)^{1/2} \left( h^2N^4  \right)^{1/2} = \frac{C \varepsilon^2 }{ h N}, 
$$
and hence the contribution of these resonant terms is $t \varepsilon^2 (hN)^{-1}$. 
We then obtain terms of the form \eqref{h1}-\eqref{h2} that can be estimated as in \eqref{bjl} but with $\delta^2$ replace by $ N^{-2\nu}$. The boundary terms are treated in a similar ways as in a similar way, and finally, we see that all the terms in \eqref{beurk2reson} can be bounded by 
 $$
| F_K^{N,h}(t) - F_K^{N,h}(0)  | \leq C\Big(\frac{ \varepsilon  N^\nu }{h } + \frac{\varepsilon^2 N^{2\nu}}{h }+ \frac{t \varepsilon^3 N^{2\nu}}{ h } + \frac{t \varepsilon^2 }{h N} \Big). 
$$
The other terms are treated similarly, and this shows \eqref{estfund2}. 

\end{document}